\documentclass[11pt,a4paper]{amsart}
\usepackage{amsmath, amssymb,amsthm}
\usepackage{color} 
\usepackage{hyperref}
\usepackage[all]{xy}
\usepackage{bbm}
\usepackage{enumerate,enumitem}
\usepackage{comment}
\usepackage{mathrsfs}
\usepackage{MnSymbol} 
\usepackage{tikz} 
\usepackage{tabularx}
\usepackage{tikz}
\usetikzlibrary{positioning}
\usepackage[T1]{fontenc}

\usepackage{graphicx}
\usepackage{float}
\usepackage{mathrsfs}
\usepackage[frenchb, main=english]{babel}
\usepackage{enumitem}

\usepackage{caption}
\newtheorem{thm}{Theorem}[section]
\newtheorem{prop}[thm]{Proposition}

\newtheorem{theoremletter}{Theorem}

\theoremstyle{definition}
\newtheorem{rem}[thm]{Remark}
\newtheorem{defn}[thm]{Definition}

\newcommand{\C}{\mathbb{C}}
\newcommand{\Q}{\mathbb{Q}}

\newcommand{\Z}{\mathbb{Z}}

\newcommand{\G}{\mathrm{G}}
\newcommand{\I}{\mathrm{I}}

\newcommand{\cB}{\mathcal{B}}
   \newcommand{\cC}{\mathscr{C}}

\newcommand{\cE}{\mathscr{E}}
   \newcommand{\cF}{\mathscr{F}}
   \newcommand{\cG}{\mathscr{G}}
\newcommand{\cI}{\mathcal{I}}
   \newcommand{\cL}{\mathscr{L}}
\newcommand{\cJ}{\mathcal{J}}
\newcommand{\cN}{\mathcal{N}}
\newcommand{\cO}{\mathcal{O}}

   \newcommand{\cS}{\mathscr{S}}
\newcommand{\cT}{\mathcal{T}}
\newcommand{\cU}{\mathcal{U}}
\newcommand{\cV}{\mathcal{V}}
\newcommand{\cW}{\mathcal{W}}
   \newcommand{\cX}{\mathscr{X}}

\newcommand{\gl}{\mathfrak{l}}
\newcommand{\gm}{\mathfrak{m}}
\newcommand{\gp}{\mathfrak{p}}

\DeclareMathOperator{\sw}{\mathsf{w}}

\DeclareMathOperator{\ad}{ad}
\DeclareMathOperator{\an}{an}
\DeclareMathOperator{\cusp}{cusp}

\DeclareMathOperator{\Frob}{Frob}

\DeclareMathOperator{\Ind}{Ind}
\DeclareMathOperator{\Gal}{Gal}
\DeclareMathOperator{\GL}{GL}  
\DeclareMathOperator{\rH}{H}

\DeclareMathOperator{\Hom}{Hom}

\DeclareMathOperator{\ord}{ord}

\DeclareMathOperator{\Spec}{Spec }

\DeclareMathOperator{\res}{res}

\DeclareMathOperator{\tr}{tr}
\DeclareMathOperator{\tor}{tor}

\newcommand{\cLm}{\mathscr{L}_{\mbox{-}}}

\title{A geometric view on Iwasawa theory}

\author{Adel Betina and Mladen Dimitrov}

\address{Univ. Vienna, Faculty of Mathematics,  Oskar-Morgenstern-Platz 1, 1090 Wien,~Austria}

\email{adelbetina@gmail.com }

\address{Univ. Lille, CNRS, UMR 8524 -- Laboratoire Paul Painlev\'e, 59000 Lille, France}
\email{mladen.dimitrov@gmail.com }

\thanks{The first author acknowledges  support from  the EPSRC (Grant EP/R006563/1) and the START-Prize Y966 of the Austrian Science Fund (FWF). The second author is partially supported by   Agence Nationale de la Recherche  grants ANR-18-CE40-0029 and ANR-16-IDEX-0004.} 

\begin{document}

\maketitle

\begin{abstract}
This article extends our study  of the geometry of the $p$-adic eigencurve at a point defined by a weight $1$ cuspform $f$  irregular at $p$ and having complex multiplication, and the implications in Iwasawa  and in Hida theories. The novel results include the determination of the Fourier coefficients  of certain  non-classical $p$-adic modular forms belonging to the generalized eigenspace of  $f$, in terms of $p$-adic logarithms of algebraic numbers. We also compute  the ``mysterious'' cross-ratios of the $p$-ordinary filtrations of the Hida families containing~$f$. 
  \end{abstract}

\addtocontents{toc}{\setcounter{tocdepth}{0}}

\section*{Introduction}

\subsection{Historical background}
In the 1960s Kubota and  Leopoldt used Kummer's congruences involving Bernoulli numbers to define the $p$-adic zeta function $\zeta_p \in \Z_p \lsem \Gal(\Q_\infty/\Q) \rsem$ and Iwasawa formulated his Main Conjecture postulating that $\zeta_p$ generates the characteristic ideal of a certain $\Z_p \lsem \Gal(\Q_\infty/\Q) \rsem$-module controlling the class group growth of the number fields contained in the cyclotomic $\Z_p$-extension $\Q_\infty$ of $\Q$. Around the same time Serre \cite{serre-LNM} observed that ordinary Eisenstein series vary  $p$-adic analytically in the weight, hence can 
be interpolated over the  $p$-adic analytic weight space   $\cW(\C_p)=\mathrm{Hom}_{\mathrm{cont}}(\Z_p^{\times}, \C_p^{\times})$, 
thus providing a modular interpretation of  the  Kubota--Leopoldt $p$-adic zeta as a constant term of a $\Z_p \lsem \Gal(\Q_\infty/\Q) \rsem$-adic Eisenstein series.

In the 1980s H.~Hida gave a new impulse to the subject by $p$-adically interpolating  in \cite{hida85}  ordinary cuspforms of weight at least $2$.
He introduced a space  of  $\Z_p \lsem \Gal(\Q_\infty/\Q) \rsem$-adic cuspform which is in a  
perfect duality of finite  free $\Z_p \lsem \Gal(\Q_\infty/\Q) \rsem$-modules with the corresponding Hecke algebra and proved a Control Theorem (see \cite{Hida86}
 when $p$ is odd and \cite{ghate-kumar} when $p=2$). Furthermore, Hida showed that every eigenform as above is the specialization of a unique, up to Galois conjugacy, 
$\Z_p \lsem \Gal(\Q_\infty/\Q) \rsem$-adic eigenform,  also called a Hida family. The Galois theoretic properties of these families, namely the restriction at the decomposition group at $p$, were studied by B.~Mazur and A.~Wiles  who also made a crucial, for us, extension of Hida Theory to  allow classical weight one forms. 

Coming back to Iwasawa theory, Eisenstein  families play a prominent role in the proof  of the  Main Conjecture over $\Q$ by Mazur and Wiles.  They  proved one divisibility using Ribet's Eisenstein ideal method and then deduced the  equality via the class number formula. In their own words,  while being `explicit' from a certain modular perspective, the approach does not allow to determine whether or  not $\zeta_p$ admits  multiple zeros  (one merely knows that if that were the case then the characteristic series of the Iwasawa module  would have a zero of the same multiplicity). Instead, they relied on the Ferrero--Greenberg Theorem \cite{ferrero-greenberg} showing that
the `trivial' zeros of $\zeta_p$ are simple. 

Another remarkable class of Hida families have complex multiplication (CM) by an imaginary quadratic field $K$, 
and  are obtained by $p$-adic interpolation of classical theta series.  Their adjoint  $p$-adic $L$-functions 
is essentially equal to a Katz anti-cyclotomic $p$-adic $L$-function,  and the corresponding anti-cyclotomic 
 Main Conjecture over $K$ has been proven by K.~Rubin and independently by H.~Hida and J.~Tilouine. Similarly to the Eisenstein case, both proofs proceed by proving one divisibility and  by invoking the class number formula, but without determining the possible zeros and their multiplicities. The question, analogous to the Ferrero--Greenberg Theorem, of whether the trivial zeros of  Katz' anti-cyclotomic $p$-adic $L$-functions are simple, has remained open 
 while being reformulated  in terms of  certain Iwasawa modules, via the mysterious bridge envisioned by Iwasawa.  In  \cite{betina-dimitrov} we  use  Hida Theory together with  Mazur's Galois Deformations Theory, to 
 show that these  trivial zeros are indeed simple, provided that  a certain anti-cyclotomic $\cL$-invariant does not vanish, 
 as  predicted by the Four Exponentials Conjecture in Transcendence Theory.

 Explicit results in Hida theory have been notoriously difficult to obtain since the Hecke-Hida algebras are finite and flat over $\Z_p \lsem \Gal(\Q_\infty/\Q) \rsem$, a complete local algebra of Krull dimension $2$ isomorphic to $\Z_p \lsem X \rsem$ by a theorem of J.P.~Serre. 
What has made our approach successful is the passage to the discrete valuation ring $\Q_p\lsem X \rsem$, and even to its strict henselization $\Lambda=\bar\Q_p\lsem X \rsem$, allowing the use of the 
rigid geometry tools  brought to the modular world by R.~Coleman and B.~Mazur in the 1990s. 
A central object in this theory is the Coleman--Mazur  eigencurve $\cC\to \cW$ first introduced in 
\cite{coleman-mazur} (under some technical assumptions which were later removed by K.~Buzzard  \cite{buzzard})  as  the rigid analytic curve parametrizing systems of Hecke eigenvalues of overconvergent $p$-adic modular forms of finite slope (see \S\ref{eigencurvebase} for a detailed presentation). 

\subsection{Intersection numbers and \texorpdfstring{$p$}{p}-adic \texorpdfstring{$L$}{L}-functions}
The local geometry at classical points has significant impact on  Iwasawa theory 
since  $p$-adic $L$-functions  have  acquired a new variable `the weight' in addition to the cyclotomic variable considered in classical Iwasawa theory. R.~Greenberg and G.~Stevens \cite{greenberg-stevens} have brilliantly illustrated  how the weight variable can be used to shed light on  problems involving {\it a priori } only 
the cyclotomic variable, such as the Mazur--Tate--Teitelbaum Conjecture \cite{MTT} on the central trivial zeros of
the  $p$-adic $L$-function of  modular forms. More recently 
J.~Bella{\"{\i}}che  \cite{bellaiche-Inv}  pushed the idea of using the weight variable even further to 
palliate the shortage of critical values needed for the construction of  $p$-adic $L$-functions for modular forms of critical slope, provided that $\cC$ is smooth at such points.

 At points where the weight map $\sw:\cC\to \cW$ is \'{e}tale, one can use the weight space $\cW$ to parametrize a Coleman family, whereas at singular points it is a challenge in itself to attach a 
$p$-adic $L$-function to a  family. 
Let us recall that  Hida's Control Theorem \cite{Hida86} for ordinary forms, extended by Coleman \cite{coleman-ocmf}  to all  forms of non-critical slope, implies that  $\sw: \cC\to \cW$ is \'{e}tale at classical non-critical $p$-regular points of weight $\geqslant 2$.  
J.~Bella\"iche and one of us  showed in \cite{bellaiche-dimitrov} that $\cC$ is smooth at classical points of weight $1$
which are $p$-regular, using Galois deformations and Transcendence Theory, namely the Baker--Brumer Theorem on the linear independence of logarithms of algebraic numbers,  to elucidate the geometry of the eigencurve. 
The geometry of $\cC$ at points which are either irregular at $p$ or have critical slope is expected to be more  complicated, and  \S\ref{geomhist} presents the  state of the art on such matters.

Another link between the  geometry of $\cC$ and Iwasawa theory, is the expectation that local intersection numbers of $\cC$ should be directly related to the adjoint $p$-adic $L$-functions, as shown in Hida's trilogy 
\cite{hida63,hida64,hida66} and vast subsequent research by numerous authors. In particular, a vanishing of the adjoint $p$-adic $L$-functions, including for trivial reasons, should detect interesting geometric phenomena. 
This has been the leitmotiv in \cite{BDPozzi}, resp.  \cite{betina-dimitrov}, where the geometry of $\cC$ at certain 
Eisenstein, resp.  CM,  weight $1$ points irregular at $p$  is related to trivial zeros of  the Kubota--Leopoldt, resp. Katz anti-cyclotomic,   $p$-adic $L$-function. 

\bigskip

\noindent 
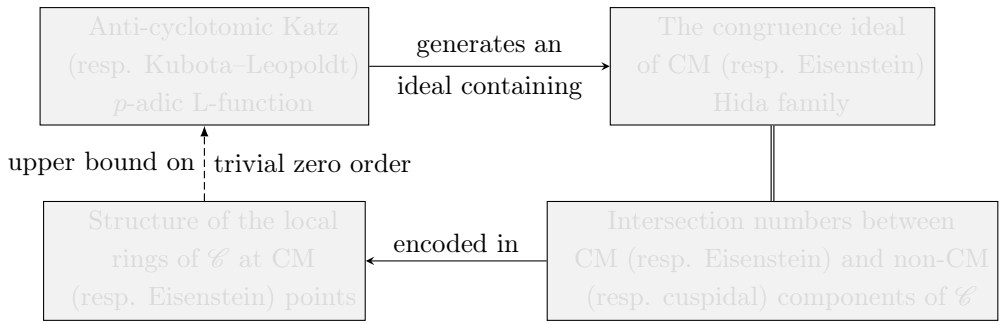
\begin{center} 
{ \small 
\begin{tikzpicture}[>=latex,rect/.style={draw=black, 
                   rectangle, 
                   fill=gray,
                   fill opacity = 0.1,
                   text opacity=1,
                   minimum width=70pt, 
                   minimum height = 25pt, 
                   align=center}, fine dots/.style={dash pattern=on 0.8pt off 0.8pt}]
]
  \node[rect] (A) {Anti-cyclotomic Katz \\ (resp. Kubota--Leopoldt)  \\ $p$-adic L-function};
  \node[rect] (B) [below=of A] {Structure of the local  \\ rings of $\cC$  at CM \\  (resp.  Eisenstein)  points};
  \node[rect, right=90pt of A] (C)  {The congruence ideal \\ of CM (resp. Eisenstein) \\ Hida family};
  \node[rect] (D) [below=of C] {Intersection numbers between \\ CM (resp. Eisenstein)  and non-CM \\  (resp. cuspidal) components of $\cC$};
  \draw[->,line width=0.15mm, dashed,dash pattern=on 1mm off 0.5mm] (B)-- node[midway, left] {upper bound on } node[midway, right] {trivial zero order} (A);
  \draw[stealth-] (B)-- node [above] {encoded in } (D);
  \draw[stealth-] (C)-- node [midway,sloped, above] {generates an }  
  node [midway,sloped, below] { ideal containing} (A);
  \draw[line width=0.15mm, double] (C)-- node [right] {} (D);
\end{tikzpicture}
}
\end{center} 

In the CM case summarized in \S\ref{CM-case}, we determine the completed local ring of $\cC$ at  the 
weight $1$ point and use the congruence ideal between CM and non-CM families  passing through that point, 
to provide an upper bound for the order of  the trivial zero of the corresponding branch $\zeta_{\varphi}^{-} \in \bar{\Z}_p \lsem X \rsem$ of the anti-cyclotomic Katz $p$-adic $L$-function. 
The exactness of this upper bound is predicted by widely accepted conjectures in Transcendence Theory. 
Furthermore, thanks to the Six Exponentials Theorem we know that at least one amongst  $\zeta_{\varphi}^{-}$  or $\zeta_{\bar\varphi}^{-}$ has a simple trivial zero.

In the Eisenstein case, treated in \cite{BDPozzi} and   summarized in \S\ref{eis-sec}, one can determine the local geometry of $\cC$ at a 
weight $1$ Eisenstein cuspidal overconvergent point unconditionally, and deduce from there Gross' formula for the derivative of the Kubota--Leopoldt $p$-adic zeta at a trivial zero. 
The non-trivial zeros all occur at  cuspidal overconvergent Eisenstein points having non-classical weight. By assuming Greenberg's pseudo-null conjecture, C.~Wang-Erickson and P. Wake showed in \cite{erickson-wake} that the cuspidal eigencurve is smooth at these points if, and only if, the zeros are simple. The question  about the \'{e}taleness of the weight map at these points is still open.

Regarding generalizations to groups of higher rank, in a collaboration with S.-C.~Shih \cite{betina-dimitrov-shih}, we  recently extended the main results of \cite{BDPozzi} to weight one Eisenstein  points on Hilbert eigenvarieties.  The investigation of Eisenstein Hida families in this setting goes at least back to Wiles'
work \cite{wiles} on the Iwasawa Main Conjecture over  totally real number fields, and also plays a prominent role in
the proof  \cite{DDP,dasgupta-kakde-ventullo} of the Gross--Stark Conjecture on the derivative of the Deligne--Ribet $p$-adic $L$-function at a trivial zero. 

\subsection{Overconvergent generalized eigenforms}
Fourier coefficients of classical eigenforms are related to arithmetic functions such as the partition function or the Dedekind eta function, and are motivic in nature as the corresponding two-dimensional Galois representations occur in the \'{e}tale cohomology of  proper smooth varieties. 
Amongst non-classical overconvergent forms, the closest in nature to a classical eigenform $f$, are those belonging 
to its generalized eigenspace $S_{\sw(f)}^{\dagger} \lsem f \rsem$. 
The very exclusive club of  genuine overconvergent generalized eigenforms $S_{\sw(f)}^{\dagger} \lsem f \rsem_0$
is a natural supplement of the classical subspace in $S_{\sw(f)}^{\dagger} \lsem f \rsem$. In his quest \cite{bellaiche-Inv,BelCM}  to attach $p$-adic $L$-functions to classical eigenforms of critical slope, J.~Bella{\"{\i}}che classified the possible such examples 
in  weight $k\geqslant 2$, and concluded that conjecturally the only genuine overconvergent generalized eigenforms 
are critical CM forms, whose Fourier coefficients were recently computed by Hsu \cite{hsu}. 
The first to take up the task in weight $1$ were H.~Darmon, A.~Lauder and V.~Rotger  \cite{DLR-Adv} 
who expressed  the Fourier coefficients of a certain  $p$-adic overconvergent weight $1$ generalized eigenform 
in terms of  $ p$-adic logarithms of algebraic numbers in ring class fields of real quadratic fields. 
The underlying classical weight $1$ form has real multiplication (RM) and, according to \cite{bellaiche-dimitrov}, it 
is the only case in which a $p$-regular weight $1$ form admits a genuine overconvergent generalized eigenspace. 
Such generalized eigenforms could provide an approach to Hilbert's twelfth problem
asking for an ``explicit Class Field Theory'' for real quadratic fields (analogous to the  theory of complex multiplication for  imaginary quadratic fields).

\medskip

In order to  state our main result concerning the $p$-irregular CM case, we need to fix some notations which will be used throughout the paper. 
We denote by $\G_L=\Gal(\bar{L}/L)$ the absolute Galois group of a  field $L$. 
The choice of an embedding $\iota_p:\bar\Q\hookrightarrow \bar\Q_p$ allows one to see 
 $\G_{\Q_p}$ as a decomposition  subgroup of $\G_{\Q}$. 
 Let $K$ be an imaginary quadratic field  having fundamental discriminant $-D$ in which  $p$ splits. 
Given a finite order character   $\psi:\G_K\to \bar\Q^\times \xrightarrow{\iota_p}  \bar\Q_p^\times $, we consider the anti-cyclotomic character 
$\varphi=\psi\cdot\bar\psi^{-1}$, where $\bar\psi$ denotes the {\it internal} Galois conjugate of $\psi$ by the complex conjugation $\tau\in \G_\Q\backslash \G_K$.  
 We assume that $\varphi_{\mid \G_{\Q_p}}$ is trivial, but $\varphi$ is not. 
Under these assumptions, the newform on $\GL_2/\Q$ obtained by automorphic induction from $K$ to $\Q$ of $\psi$ 
is a weight $1$ theta series  $\theta_\psi$ of level $N$ and central character $\varepsilon=\varepsilon_K\cdot \psi\circ\mathrm{Ver}$. Here  
$\varepsilon_K$ denotes the quadratic Dirichlet character attached to $K/\Q$, and 
$\mathrm{Ver}$ denotes the transfer homomorphism. 
Let $f=\sum_{n\geqslant 1} a_n q^n$  denote the unique $p$-stabilization of $\theta_\psi$
(see \eqref{qexp-f}). 

Throughout the paper, we make the following assumption on 
the  anti-cyclotomic $\cL$-invariants introduced in   \cite[\S1]{betina-dimitrov}
\begin{equation}\label{generic-CM}
\cLm(\varphi) \cdot \cLm(\bar\varphi) \cdot (\cLm(\varphi) + \cLm(\bar\varphi)) \ne 0, 
\end{equation}
and we choose a square root $\xi$ of $\cLm(\bar\varphi)\cLm(\varphi)^{-1}\cS_{\varphi}^{-1}$,
where the slope $\cS_{\varphi}$ is defined  in  \eqref{slope-def}. 
One should observe  that \eqref{generic-CM} is true for  $\varphi$  quadratic,  whereas in all other cases 
the Schanuel Conjecture  predicts that $\cLm(\varphi)$ 	and $\cLm(\bar\varphi)$  are algebraically independent. 
Finally,  we consider  the $p$-adic logarithms $\cL_{\gl}$ (for $\ell\ne p$ splitting in $K$ as 
$\gl\bar\gl$)  and $\cL_{\psi,\ell}$ (for $\ell$ ramified or inert in $K$)  of some explicit $\ell$-units defined in  \eqref{defn-ellunit} and  \eqref{psi-invariant}, respectively.

\begin{theoremletter}\label{main-q-exp}
There exists a basis $\{f^\dag_{\cF},f^\dag_{\Theta}\}$ of $S_{\sw(f)}^{\dagger} \lsem f \rsem_0$ whose $q$-expansion is as follows:
 \begin{enumerate}
\item For any prime $\ell \ne p $ splitting in $K$ as  $\gl \cdot \bar{\gl}$, one has 
\[ a_\ell(f^\dag_{\cF})=0 , \,\,  
a_\ell(f^\dag_{\Theta})= (\cL_{\gl} - \cL_{\bar{\gl}}) \cdot(\psi(\gl) -\psi(\bar\gl)). \]
 \item For any prime $\ell \mid DN$ not splitting in $K$,  one has  
\[  a_\ell(f^\dag_{\Theta})=0 , \,\,   a_\ell(f^\dag_{\cF})=
2 \psi(\gl) \xi   \cL_{\psi,\ell}\cdot \frac{\cLm(\varphi)}{\cLm(\varphi)+\cLm(\bar\varphi)}. \] 

\item For any prime $\ell\nmid N$ inert in $K$,  one has 
  \[ a_\ell(f^\dag_{\Theta})=0 , \,\,  \text{and } \,    a_\ell(f^\dag_{\cF})=
 2  \xi    \cL_{\psi,\ell}  \cdot \frac{\cLm(\varphi)}{\cLm(\varphi)+\cLm(\bar\varphi)}. \] 
\item Any form $\sum_{n \geqslant 1} a_n^\dagger\cdot q^n\in S_{\sw(f)}^{\dagger} \lsem f \rsem_0$ satisfies $a_1^\dag=a_p^\dag=0$ and the  following recursive relations: 
\begin{align*} 
 a_{mn}^\dag&=a_m a_n^\dag+a_n a_m^\dag,
 \text{ for all   } (n,m)=1, \text{ and } \\
 a_{\ell^r}^\dag&= \begin{cases} a_\ell a_{\ell^{r-1}}^\dag+a_{\ell^{r-1}}a_{\ell}^\dag -
 \varepsilon(\ell) a_{\ell^{r-2}}^\dag& , \text{ for  all primes }  \ell \nmid Np \text{ and  all } r\geqslant 2, \\
 r a_\ell^{r-1}a_{\ell}^\dag&,  
 \text{ othewise}. 
  \end{cases} 
\end{align*} 
\end{enumerate}
\end{theoremletter}

Most of the paper is organized around the proof of the above Theorem. In \S\ref{eigencurvebase} we summarize the basic properties of $\cC$ while making a detour to define a geometric $q$-expansion of a  Coleman family 
at a cusp of the ordinary locus using the overconvergent modular sheaf constructed by V.~Pilloni \cite{pilloni}. 
 Using the resulting  $q$-expansion Principle we prove a perfect Hida duality between the space of Coleman families and the corresponding Hecke algebra. Exploiting this duality and  the results of \cite{betina-dimitrov} 
on the local geometry of $\cC$ at $f$, allows us  in \S\ref{sec-proofs} to  compute infinitesimally 
the Fourier coefficients of all families passing through $f$ and to obtain Theorem \ref{main-q-exp}. The resulting formulas involve 
$p$-adic logarithms of algebraic numbers in the field cut out by the adjoint Galois representation attached to  $f$.

\subsection{A mysterious cross-ratio}
Let us first formulate the precise problem. We let $f$ be a $p$-ordinary stabilization of a newform of weight $k\geqslant 1$ and level 
$\Gamma_1(N)$, and denote by $\alpha\ne 0$ its $U_p$-eigenvalue.  If $f$ is $p$-regular then 
one knows that  (see \S\ref{weight-geq-2} for more details)
\begin{itemize}
\item $f$ belongs to a unique, up to Galois conjugacy, Hida family $\cF$ and the $\Lambda=\cO_{\cW,\sw(f)}^\wedge\simeq\bar\Q_p\lsem X \rsem$-algebra $\cO_{\cC,f}^\wedge$ is isomorphic to $\bar\Q_p\lsem Y \rsem$, where $Y^e=X$ for some $e\geqslant 1$,  
\item the corresponding  Galois representation $\rho_{\cF}:\G_\Q\to \GL_2(\bar\Q_p\lsem Y \rsem)$ is $p$-ordinary and its 
$\G_{\Q_p}$-stable line reduces modulo $(Y)$ to the unique $\G_{\Q_p}$-stable line in $\rho_f$ such that the
arithmetic Frobenius $\Frob_p$ acts by $\alpha$ on the (unramified) quotient. 
\end{itemize}

If $f$ is irregular  at $p$ and $k=1$, then the restriction of $\rho_f$ to $\G_{\Q_p}$ is scalar, given by an unramified character sending $\Frob_p$ to $\alpha$. Galois conjugacy classes of Hida families $\cF$ containing $f$   are in bijection with the irreducible components of $\cC$ containing $f$, and 
\begin{itemize}
\item either $\cO_{\cF,f}^\wedge$ is not a regular ring ({\it e.g.} it is not even normal) and  $\rho_{\cF}$ might not even admit an ordinary filtration, 
\item or $\cO_{\cF,f}^\wedge\simeq\bar\Q_p\lsem Y \rsem$ is a discrete valuation ring and thus $\rho_{\cF}$ does admit a (unique) ordinary filtration yielding, when reduced modulo $(Y)$,  a well defined line in $\rho_f$, {\it i.e.} an element of $\mathbb{P}(\rho_f)$. 
\end{itemize}

The choice of basis for $\rho_f$ allows one to identify  $\mathbb{P}(\rho_f)$ with $\mathbb{P}^1(\bar\Q_p)$ and each of the 
 finitely many ``regular'' Hida families containing  $f$ picks a well-defined element in it. 
 We will now illustrate this phenomenon with two familiar   examples.

Suppose first that  $f$ is a weight $1$ Eisenstein series which is irregular at $p$. It is  natural to chose a basis in which $\rho_f$ is reducible and semi-simple. 
There are two Eisenstein Hida families containing $f$ having residual slope $0$ and  $\infty$, respectively. 
The main result in \cite{BDPozzi} shows that there is a unique cuspidal Hida family $\cF$ containing $f$, 
whose residual slope belongs to $\bar\Q_p^\times$. Since one can rescale the vectors of the  basis, 
all values in $\bar\Q_p^\times=\mathbb{P}^1(\bar\Q_p)\setminus \{0,\infty\}$ are allowed, 
the forbidden values $0$ and $\infty$ corresponding to the  two $\G_\Q$-stable lines. 
One can then recover $\rho_{\cF}$ as the universal ordinary deformation of $\rho_f$ endowed with {\it any } 
non-$\G_\Q$-stable line. A tame analogue of such deformation problems  was used by 
F.~Calegari and M.~Emerton in \cite{calegari-emerton} to establish an $R=T$ theorem for the
weight $2$ Hecke algebra at Eisenstein primes.

Assume now  that $f$ is a weight $1$ cuspform irregular at $p$ and having complex multiplication. 
The situation is then more rigid as $\rho_f$ is irreducible. As $\rho_f$ is odd, a canonical pair of elements of $\mathbb{P}(\rho_f)$ is given by the eigenspaces for the complex conjugation $\tau\in \G_\Q$. 
Recall the notations introduced immediately before Theorem \ref{main-q-exp}, in particular the number 
$\xi\in \bar\Q{}_p^\times$. 

\begin{theoremletter}\label{thm-cross-ratio}
Under the assumption  \eqref{generic-CM}, the ordinary lines of the four families containing $f$  are pairwise distinct and 
their  cross-ratio belongs to  $\{-1,2,\frac{1}{2}\}$. Moreover, the cross-ratio of the line fixed by the complex conjugation 
and the three lines in $\rho_f$ obtained by reducing the  ordinary lines of $\cF$ and the two CM families containing $f$, 
belongs to 
$\left\{\xi,\frac{1}{\xi}, 1-\xi, \frac{1}{1-\xi},\frac{\xi}{\xi-1},\frac{\xi-1}{\xi} \right\}$.
\end{theoremletter}
 
 
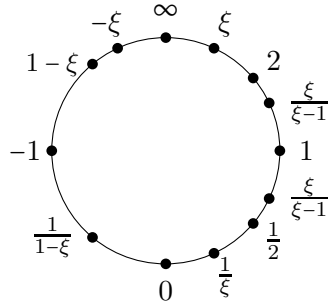
\vspace{-5mm}
\begin{center}
\begin{tikzpicture}
  \coordinate (center) at (1,2);
  \def\radius{1.5cm}
  \draw (center) circle[radius=\radius];
   \fill[] (center) ++(0:\radius) circle[radius=2pt]++(0:10pt) node {$1$};
  \fill[] (center) ++(180:\radius) circle[radius=2pt]++(180:10pt) node {$-1$};
   \fill[] (center) ++(270:\radius) circle[radius=2pt]++(270:10pt) node {$0$};
  \fill[] (center) ++(90:\radius) circle[radius=2pt] ++(90:10pt) node {$\infty$};
    \fill[]   (center) ++(115:\radius) circle[radius=2pt] ++(115:10pt) node {$-\xi$};
     \fill[] (center) ++(65:\radius) circle[radius=2pt] ++(65:10pt) node {$\xi$};
          \fill[] (center) ++(-65:\radius) circle[radius=2pt] ++(-65:10pt) node {$\tfrac{1}{\xi}$};
   \fill[] (center) ++(40:\radius) circle[radius=2pt] ++(40:10pt) node {$2$};
      \fill[] (center) ++(-40:\radius) circle[radius=2pt] ++(-40:10pt) node {$\tfrac{1}{2}$};
   \fill[] (center) ++(25:\radius) circle[radius=2pt] ++(0:15pt) node {$\tfrac{\xi}{\xi-1}$};
  \fill[]   (center) ++(-25:\radius) circle[radius=2pt] ++(0:15pt) node {$\tfrac{\xi}{\xi-1}$};
      \fill[]   (center) ++(130:\radius) circle[radius=2pt] ++(180:15pt) node {\small $1-\xi$};
      \fill[]   (center) ++(-130:\radius) circle[radius=2pt] ++(180:15pt) node {$\tfrac{1}{1-\xi}$};
\end{tikzpicture}
\end{center}

\vspace{-5mm}
\tableofcontents
\addtocontents{toc}{\setcounter{tocdepth}{1}}

\section{Background on the \texorpdfstring{$p$}{p}-adic eigencurve  \texorpdfstring{$\cC$}{C}}\label{eigencurvebase}

\subsection{Overconvergent modular forms} \label{ocmf}
Let $p$ be any prime number. 
For  an integer $N \geqslant 4$ relatively   prime to $p$, we let $\cX$ be the proper smooth modular curve of level $\Gamma_1(N)$ over $\Z_p$ and $\cE \to \cX$ be the universal generalized elliptic curve. The fiber of $\cE$ above any cusp is given by a certain N\'eron polygon endowed with  $\Gamma_1(N)$-level structure. 

The invertible sheaf $\omega$ on $\cX$ is defined as the pull-back of the relative differentials $\Omega_{\cE/\cX}$ along the  zero section of $\cE \to \cX$.  The space of classical modular forms of weight $k \in \Z_{\geqslant 1}$, level $\Gamma_1(N)$ and coefficients in a $\Z_p$-algebra $A$ is defined as 
$M_k(N;A)=\rH^0(\cX_{A}, \omega^{\otimes k}_{A})$. The $A$-module $M_k(N;A)$ is functorial in $A$ and commutes with flat base change.

Let $X^{\an}$ be the rigid analytification of  the generic fibre $X=\cX_{\Q_p}$ of $\cX$. 
 Note that  the properness of $\cX$ over $\Z_p$ implies that $X^{\an}$
 is also the rigid space  in the sense of Raynaud. The analytification of the line bundle $\omega$  is a line bundle on $X^{\an}$ and will still be denoted  by $\omega$. 
 
 The  ordinary locus $X^{\ord}$ is  the complement of the supersingular residue classes  in $X^{\an}$ and  can be characterized as the locus where  
  the truncated valuation of the Hasse invariant is $0$. More generally, for $v\in \Q_{>0}$,  let $X(v)$ denote the strict overconvergent neighborhood 
  in $X^{\an}$  of the ordinary locus $X^{\ord}=X(0)$  where the (truncated) valuation of the Hasse invariant is $ \leqslant v$.  The 
  space  Katz $p$-adic  modular forms of weight $k\in \Z$ is the infinite dimensional $\C_p$-vector space 
  $\rH^0(X^{\ord}_{\C_p}, \omega^{\otimes k}_{\C_p})$, whereas the space of $p$-adic overconvergent modular forms was defined  by Coleman as the 
    Fr\'echet $\C_p$-vector space: 
  \[M^\dagger_k=\lim_{\substack{\longrightarrow\\ v >0}}  \rH^0(X(v)_{\C_p} ,\omega^{\otimes k}_{\C_p}).\] 

The weight space   $\cW_p$ is the rigid space over $\Q_{p}$ representing  homomorphisms $\Z_{p}^{\times} \rightarrow \mathbb{G}_{m}$. 
We consider  $\Z$  as a subset of $\cW_p$  by sending $k\in \Z$ to the algebraic character $x\mapsto x^{k}$. 
The space $\cW_p$ is a disjoint union, indexed by the characters of $(\Z/2p\Z)^{\times}$, of copies of  the rigid open unit disk
$\{\mid z-1 \mid_p <1 \}$  representing  homomorphisms $1+2p\Z_{p} \rightarrow \mathbb{G}_{m}$
(via the image of the topological generator $1+2p$). 
The latter is admissibly covered by the closed disks $\cB_m=\{\mid z-1 \mid_p\leqslant p^{-1/n_m}\}$, where the  increasing sequence of positive integers  $(n_m)_{m\geqslant 1}$ is chosen so that the following holds. 
For  any  $m\in \Z_{\geqslant 1}$ there exists a (unique) character 
\[\tilde{\kappa}_m:\Z_p^{\times}\cdot(1+p^m  \cO_{\C_p}) \to \cO(\cB_m)^{\times}\] 
extending the universal character $\kappa_m: \Z_p^{\times} \to \cO(\cB_m)^{\times} $ and whose restriction to $(1+p^m  \cO_{\C_p})$ is analytic.

The invertible sheaf $\omega$ is representable by $\Hom_{\cX}(\cO_{\cX}, \omega)$ and we 
denote by $\pi: \cI=\mathrm{Isom}_{\cX}(\cO_{\cX}, \omega)\to \cX$  the  corresponding $\mathbb{G}_m$-torsor. 
The fibers of the  rigid analytification  $\pi^{\an} :\cI^{\an} \to X^{\an}$ are naturally isomorphic to  $\C_p^{\times}$. 
For  $k \in \Z$ one can recover $\omega^k$ as $\pi^{\an}_{*}(\cO_{\cI^{\an}})[k]$, where 
$[k]$ means the $k$-equivariant sections for the action of $\mathbb{G}_m$.

In \cite[\S3]{pilloni},  Pilloni showed that there exists  $v_m>0$ and an invertible sheaf $\omega_m $ on $X(v_m) \times \cB_m$
specializing,  for any  $k\in \Z_{\geqslant 1} \cap \cB_m$, to the automorphic line bundle  $\omega^{\otimes k}$ on $X(v_m)$. 
Namely, he constructed an open $\cJ_m$   of $\cI^{\an}_{\mid X(v_m)}$  endowed with $\Z_p^{\times}$-action   such that 
\[\omega_m  = \left(\pi^{\an} _*\cO_{\cJ_m} \widehat{\otimes} \cO_{\cB_m} \right)[\kappa_m ],\]
where  $[\kappa_m ]$ means the $\kappa_m $-equivariant sections for the action of $\Z_p^{\times}$.
More precisely, by {\it loc. cit.}  $\cJ_m$ is locally isomorphic for the \'{e}tale topology on $X^{\ord} $  to the union of disks 
\[\Z_p^{\times}\cdot (1+p^m  \cO_{\C_p})=\underset{y \in (\Z/p^m\Z)^{\times}}{\bigcup}\mathrm{B}(\tilde{y},p^{-m}) \subset \C_p^{\times},\] where    
$\tilde{y}\in \Z_p^{\times} $ denotes a  lift of $y$  and $ \mathrm{B}(\tilde{y},p^{-m}) =\{z\in \C_p , \mid z-\tilde{y} \mid_p\leqslant p^{-m}\}$. 
By definition $\omega_m$  is locally isomorphic for the \'{e}tale topology on $X^{\ord}$ to the eigenspace of $\cO(\cB_m)$-valued locally analytic functions on $\Z_p^{\times}\cdot(1+p^m  \cO_{\C_p})$ which are $\kappa_m$-equivariant for the action of $\Z_p^{\times}$. This space is clearly generated by the section corresponding to $\tilde{\kappa}_m$. 

Let $\cU$ be an open   affinoid  of  $\cW_p$ which is  contained  in some $\cB_m$. 
For $v>0$ sufficiently small, we  let $\omega_\cU$ denote the corresponding invertible sheaf on $X(v) \times \cU$. 
The correspondence $U_p$ is  defined  on the locus $X\left(\tfrac{p}{p+1}\right)$ where the canonical subgroup of $\cE$ exists.  It sends $X(v)$ on $X\left(\frac{v}{p}\right)$ and the resulting endomorphism of the
$\cO(\cU)$-Banach module $\rH^0(X(v) \times \cU,\omega_\cU)$ is compact. 
As $\cU$ is reduced, after possibly shrinking it around any given point, the 
$\cO(\cU)$-Banach module $\rH^0(X(v) \times \cU,\omega_\cU)$ admits  slope decomposition (see \cite[\S2.3]{hansen}). 
It follows that the subset $\rH^0(X(v) \times \cU,\omega_\cU)^{\leqslant s}$ 
of elements of  slope at most $s\in \Q_{\geqslant 0}$ is a $\cO(\cU)$-submodule which is 
locally free and of finite type.   The $\cO(\cU)$-module of  Coleman families is defined as
\[ M^{\dag,\leqslant s}_\cU= \varinjlim_{v>0} \rH^0(X(v) \times \cU,\omega_\cU)^{\leqslant s}. \]

The weight space $\cW_p$  can be admissibly covered by countably many open affinoids $\cU$ as above for each of which  
$M^{\dag,\leqslant s}_\cU$ is free of finite type over $\cO(\cU)$. 

\begin{rem}\label{radius}
One can extend the definition to  tame levels $N \leqslant 3$ by considering the analogous objects on the modular curve of level $\Gamma_1(4N)$ and then taking  $\Gamma_1(N)$-invariants. Note however that this will not be needed in \S\ref{sec-proofs}
where Theorems \ref{main-q-exp} and \ref{thm-cross-ratio} are proven, as $N\geqslant D \geqslant 4$ in that case. 
\end{rem}

\subsection{$q$-expansions of Coleman families} 
In this subsection we describe how one  can attach to a Coleman family a geometric $q$-expansion which interpolates the $q$-expansions of its classical specializations and satisfies the $q$-expansion Principle. 
 
The generalized elliptic curve $\mathrm{Tate}(q)$ over $\Z_p\lsem q \rsem$, with $1$-gon special fiber and natural $\Gamma_1(N)$-level structure defines a morphism $\Spec(\Z_p\lsem q \rsem) \to \cX$ (the section $q=0$ corresponds to $\infty$).    Let $d^\times t$ be the canonical differential of $\mathrm{Tate}(q)$,  {\it i.e.},   a canonical $\cO_{\cX,\infty}^\wedge$-basis of the completed stalk ${\omega}_{\infty}^\wedge$ of $\omega$ along the section  $\infty:\Spec(\Z_p) \to \cX$. For $k\in \Z$, one can identify $(\omega_{\infty}^\wedge)^{\otimes k}$ with 
$\cO_{\cX,\infty}^\wedge=\Z_p \lsem q \rsem$ using the canonical basis $(d^\times t)^{\otimes k}$, allowing one to geometrically define the  $q$-expansion of an overconvergent modular form.

\begin{prop}\label{q-expprinciple}There exists a neighborhood $\cV$ for the \'{e}tale topology of the cusp $\infty \in X^{\ord}$ and a generator  of $\omega_\cU(\cV \times \cU)$ specializing at any $k\in \Z_{\geqslant 1} \cap \cU$ to the canonical differential $(d^\times t)^{\otimes k}$. In particular, one can  attach to any Coleman family $\cF \in M^{\dag,\leqslant s}_\cU$ a $q$-expansion $\sum_{n  \geqslant 0} a_n(\cF) q^n \in \cO(\cU) \lsem q \rsem$ interpolating the $q$-expansions of its classical specializations. Moreover,  the  $q$-expansion Principle holds,  {\it i.e.},  the $q$-expansion  map is injective.  
\end{prop}

\begin{proof} As already observed,  $\omega_{\cU}$  is locally isomorphic for the \'{e}tale topology on $X^{\ord}\ni \infty $ to the eigenspace of $\cO(\cU)$-valued locally analytic functions on $\Z_p^{\times}\cdot(1+p^m  \cO_{\C_p})$ which are $\Z_p^{\times}$-equivariant with respect to the action of the
universal character $\kappa_\cU: \Z_p^{\times} \to \cO(\cU)^{\times}$, a basis being given by the locally analytic character 
$\tilde{\kappa}_\cU$. Here $m$ is chosen so that $\cU\subset \cB_m$. 

Since $\mathrm{Tate}(q)$ is ordinary at $p$, one can choose a neighborhood $\cV$ for the \'{e}tale topology of the cusp $\infty$ with a local trivialization of $\omega$ given by $d^\times t$ and such that the section $\tilde{\kappa}_{\cU}$ generates $\omega_\cU(\cV \times \cU)$. The specialization of  $\tilde{\kappa}_{\cU}$ at any $k \in \Z_{\geqslant 1} \cap \cU$ corresponds under this construction to the canonical differential $(d^\times t)^{\otimes k} $ generating $\omega^{\otimes k}(\cV)$, thus providing the sought-for  $p$-adic  analytic interpolation. This yields the desired trivialization $\omega_\cU(\cV \times \cU) \simeq \cO(\cV) \widehat{\otimes} \cO(\cU) $ together with the natural injection $\cO(\cV) \widehat{\otimes} \cO(\cU) \hookrightarrow \cO(\cU) \lsem q \rsem$ given by the localization $\cO(\cV) \to \cO^\wedge_{X^{\ord},\infty}=\bar\Q_p \lsem q \rsem$ at $\infty$. In this manner we have associated to any Coleman family $\cF \in M^{\dag,\leqslant s}_\cU$ a $q$-expansion $\sum_{n  \geqslant 0} a_n(\cF) q^n \in \cO(\cU) \lsem q \rsem$ interpolating the $q$-expansions of its classical specializations, and thus satisfying a  $q$-expansion Principle. 
\end{proof}

\begin{rem} Using similar techniques one can prove a statement analogous   to  Proposition \ref{q-expprinciple}
at an arbitrary cusp of $X^{\ord}$. 
\end{rem}
\subsection{Basic global properties of $\cC$}
The eigencurve $\cC$ of tame level $N$ is admissibly covered   by  the affinoids attached to the  $\cO(\cU)$-algebras $\cT_{\cU}^{\leqslant s}$ generated by the Hecke operators   $T_\ell$, $\langle \ell\rangle$, $ \ell\nmid Np$ and $U_p$
acting on $M^{\dag,\leqslant s}_\cU$, where  $s\in \Q_{\geqslant 0}$ is arbitrary and  the  open affinoids $\cU$ form an admissible cover of $\cW_p$ as in \S\ref{ocmf}. 

Henceforth we will use the weight space $\cW$ representing the continuous homomorphisms: 
\[\Z_{p}^{\times}\times (\Z/N\Z)^\times \rightarrow \mathbb{G}_{m},\]
which is endowed with {\it shifted} forgetful map to $\cW_p$ and is locally generated over the latter by the diamond operators $\langle a\rangle$, 
$a\in(\Z/N\Z)^\times$. The shift is made so that $k\in\Z$ henceforth corresponds to the character $x\mapsto x^{k-1}$ of 
$\cW$,  and on the level of Iwasawa algebras is given by the automorphism of 
$\Z_p\lsem 1+2p\Z_p \rsem$ sending $[1+2p]$ to $(1+2p)[1+2p]$.

The eigencurve $\cC$ is reduced, and it follows from its construction that there exists a flat and locally finite morphism $\sw: \cC \rightarrow \cW$, called the weight map. Moreover $\sw$ is proper by \cite{diao-liu}. 
Thanks to the above shift, the classical weight $1$ forms, which are the focal point of our study, 
map under $\sw$ to finite order characters equal to the determinant of the corresponding Galois representation (pre-composed with the Artin reciprocity map). 

By construction of $\cC$, there exist bounded  global sections $\{ T_{\ell},U_{p}\}_{\ell \nmid Np} \subset \cO_{\cC}^{+}(\cC)$ such that the usual application ``system of eigenvalues'' 
\[ x \in \cC(\bar\Q_{p})  \mapsto \{T_{\ell}(x),U_{p}(x)\}_{\ell \nmid Np}\] 
is injective, and produces all systems of eigenvalues for $\{ T_{\ell},U_{p}\}_{\ell \nmid Np}$ acting on the space of overconvergent forms  with  coefficients in $\bar \Q_p$, of tame level $N$,  having weight in $\cW(\bar\Q_p)$ and  a non-zero $U_{p}$-eigenvalue.

A fundamental arithmetic tool in the study of the geometry of $\cC$ is the  universal $2$-dimensional pseudo-character  
\begin{equation}\label{pseudo-char}
\tau_{\cC}:\G_{\Q} \rightarrow  \cO_{\cC}(\cC),
\end{equation}
which is unramified at all $\ell\nmid Np$ and 
such that $\tau_{\cC}$ maps an arithmetic Frobenius  $\Frob_\ell$ to $T_{\ell}$. 
 This pseudo-character interpolates $p$-adically the traces of semi-simple $p$-adic Galois representations attached to the classical points of $\cC$. While these Galois representations are De Rham at $p$, the semi-simple $p$-adic Galois representation attached to an arbitrary specialization of $\tau_{\cC}$  is only trianguline at $p$.

\subsection{Classical points of $\cC$}
A point of  $\cW$ is said to be classical if its restriction to some open subgroup of $\Z_p^\times$ 
 is given by the homomorphism $(x\mapsto x^{k-1})$,
for some  $k\in \Z_{\geqslant 1}$ (such characters are  locally algebraic). 
The subset  $\cW^{\mathrm{cl}} \subset \cW(\bar\Q_p)$ of classical weights
 is very Zariski dense  in the sense that for any affinoid $\cU$ of $\cW$, the intersection 
  $\cW^{\mathrm{cl}} \cap \cU$  is either empty or is Zariski dense in $\cU$.  A classical point of $\cC$ always maps  to a point of $\cW^{\mathrm{cl}}$, but the converse in not necessarily true. 
  
In order to describe the classical points of $\cC$ let us first recall the notion of a $p$-stabilization. 
Let $f(z)=\sum_{n \geqslant 0 } a_n e^{2i\pi nz}$  be a primitive normalized eigenform 
 of weight $k\in \Z_{\geqslant 1}$,   central character 
$\varepsilon$, and  level $\Gamma_1(Mp^t)$, with $t \in \Z_{\geqslant 0}$ and $M$ dividing $N$. We distinguish the following two cases.  
\begin{enumerate}
\item If $t=0$, then the Hecke polynomial $X^2-a_pX+\varepsilon(p)p^{k-1}$ of $f$ at $p$ has two (necessarily non-zero, but not necessarily distinct) roots, denoted 
 $\alpha$ and $\beta$. The corresponding $p$-stabilizations $f_\alpha(z)=f(z) - \beta f(pz)$ and 
 $f_\beta=f(z) - \alpha f(pz)$ both have   level $\Gamma_1(M)\cap \Gamma_0(p)$ and define points $f_\alpha$ and   $f_\beta$ in $\cC^{\mathrm{cl}}$.  If those points are distinct, we call them $p$-regular, if not,  $p$-irregular. By an abuse of language we sometime say that $f$ itself is regular or irregular at $p$.
 
 \item If $t > 0$, then $f$ is already a $U_p$-eigenvector with eigenvalue $\alpha=a_p$.  If $f$ has finite slope ({\it i.e.},  $\alpha\ne 0$), then it defines a  point $f_\alpha=f \in\cC^{\mathrm{cl}}$. As the Hecke polynomial equals $X(X-\alpha)$, the point 
 $f$ is regular at $p$, . 
\end{enumerate} 

The set of classical points $\cC^{\mathrm{cl}}\subset \cC(\bar\Q_p)$ consists of all 
 points $f_\alpha$ as above (considered with coefficients in $\bar\Q_p$ via $\iota_p$)  as $k\in \Z_{\geqslant 1}$, $t\in \Z_{\geqslant 0}$ and   $M$ dividing $N$ vary. By Coleman's Control Theorem $\cC^{\mathrm{cl}}$ is very Zariski dense  in $\cC$.

The $p$-adic valuation of the $U_p$-eigenvalue of a point in $\cC(\bar\Q_p)$ is called its {\it  slope}. The slope of a classical weight $k$ point cannot exceed $k-1$ and is called {\it critical} when the equality is reached. 
The points having slope $0$ are called {\it ordinary}. The locus $\cC^{\ord}$ of $\cC$ where $\mid U_{p}\mid_p  = 1$ is open and closed in $\cC$, and  is called the ordinary locus. A formal model of $\cC^{\ord}$ is  given by the universal $p$-ordinary reduced Hida Hecke algebra of tame level $N$ generated by the Hecke operators $T_\ell$ for all primes $\ell \nmid Np$ and $U_p$.

\subsection{Hecke operators at primes dividing the level}
For each $\ell\mid N$ the module $M^{\dag,\leqslant s}_\cU$ is endowed with a $\cO(\cU)$-linear operator $U_\ell$, commuting with  $\cT_\cU^{\leqslant s}$. 
It can be either defined  geometrically as a correspondence between modular curves of levels prime to $p$, or by the usual formulas on $q$-expansions,  {\it i.e.},  using $p$-adic interpolation of classical forms. By adding those operators to $\cT_\cU^{\leqslant s}$ one can define the full eigencurve $\cC^{\mathrm{full}}$, whose ordinary part is directly related to the Hida Hecke algebras in their most classical definition. The  advantage of working with the full Hecke algebra would become transparent in \S\ref{sec:hida-duality}. The disadvantage is that this bigger algebra is not necessarily  reduced. Luckily one does not have to choose when working in a neighborhood of a classical point corresponding to newform,  as the next proposition shows that the two are locally isomorphic

\begin{prop}\label{bad-Hecke-ops}
Any $f\in\cC^{\mathrm{cl}}$ corresponding to a  newform of tame level $N$  has an affinoid neighborhood $\cV$ such that 
any $g\in \cV\cap \cC^{\mathrm{cl}}$ corresponds to a  newform of tame level $N$. Moreover, $\cC$ and $\cC^{\mathrm{full}}$ are  isomorphic locally at $f$, in particular $\cC^{\mathrm{full}}$ is reduced locally at $f$. 
\end{prop}
\begin{proof} As $\rho_f$ is irreducible, by standard arguments (see for example \cite[Propositions~5.1]{bellaiche-dimitrov}) there exists an affinoid neighborhood $\cV$ of $f$ in $\cC$ and a continuous representation $\rho_{\cV}: \G_{\Q} \rightarrow  \GL_2(\cO(\cV))$ whose trace equals the pseudo-character  $\tau_{\cV}:\G_{\Q} \rightarrow  \cO(\cV)$ obtained from \eqref{pseudo-char}.

Fix a prime  $\ell$   dividing  $N$ and recall that $f$ is new at $\ell$. As the local and global Langlands correspondences for $\GL(2)$ are compatible, in order to show that $g\in \cV\cap \cC^{\mathrm{cl}}$ is new at $\ell$ as well, it suffices to show that the restrictions of  $\rho_{f}$ and $\rho_{g}$ to the inertia subgroup $\I_\ell$ at $\ell$ are isomorphic. To perform this part of the argument
we may restrict our study to the irreducible component  of $\cV$ containing $g$, {\it i.e.} we may and do temporarily assume that 
$\cV$ is irreducible. 
Let  $(r_\cV, \cN_\cV)$ be  the Weil-Deligne representation attached to $\rho_{\cV\mid \G_{\Q_\ell}}$ in \cite[Lemma~7.8.14]{bellaiche-chenevier-book}, and similarly let
$(r_g, \cN_g)$ denote  the Weil-Deligne representation attached to $\rho_{g\mid \G_{\Q_\ell}}$. 
By \cite[Lemmas~7.8.17]{bellaiche-chenevier-book} one knows that 
$\tau_{\cV\mid\I_\ell}$ and $r_{\cV\mid\I_\ell}$ are constant over $\cV$, isomorphic to  
$\tr(\rho_{f})_{\mid\I_\ell}$ and $r_{f\mid\I_\ell}$,  respectively. 

If $\cN_\cV=0$ then $\rho_{\cV}(\I_\ell)=r_{\cV}(\I_\ell)$ is finite and isomorphic to $\rho_f(\I_\ell)$. 

If $\cN_\cV\ne 0$ then,  after possibly shrinking the irreducible affinoid $\cV$, one can assume that  $\cN_\cV$ does not vanish over $\cV\backslash\{f\}$. 
It follows that any  classical $g\ne f$ in $\cV$  is given  at $\ell$ by the Steinberg representation twisted by a character $\chi$. 
One deduces then from  $\tr(\rho_{f})_{\mid\I_\ell}=2\chi_{\mid\I_\ell}$ that 

$\bullet$ either $f$ at $\ell$ is the Steinberg representation twisted by a character having the same restriction to $\I_\ell$ as $\chi$, 
in which case   $\rho_{g|\I_\ell}$ and $\rho_{f|\I_\ell}$ are isomorphic, 

$\bullet$ or $f$ at $\ell$ is a principal series attached to two characters $\chi_1$ and $\chi_2$ having the same restriction to $\I_\ell$ as $\chi$. By continuity of $\rho_{\cV\mid \G_{\Q_\ell}}$ and using the  density of  such $g$ in $\cV$, one deduces 
that $(\chi_1/\chi_2)(\Frob_\ell)=\ell^{\pm 1}$. This is impossible as the cuspform 
$f$  satisfies  the Ramanujan Conjecture, proved by P.~Deligne. 

So far we have proven that the restriction $r_{g\mid\I_\ell}$  and the rank of $\cN_g$ are 
both constant as $g$ varies over the classical points in a neighborhood $\cV$ of $f$ in $\cC$. 
In particular, all  points of $\cV^{\mathrm{cl}}$ are new at all primes $\ell\ne p$. 
By Coleman's Control Theorem  and the Strong Multiplicity One Theorem for $\GL(2)$  it follows that for any 
$g\in \cV\cap \cC^{\mathrm{cl}}$ of non-critical slope, the corresponding generalized eigenspace in $M_{\sw(g)}^{\dagger, \leqslant s}$ is one dimensional, generated by $g$.

We consider the affinoid neighborhood $\cU=\sw(\cV)$ of $ \sw(f)$ in $\cW$, and the affinoid neighborhood $\cV^{\mathrm{full}}$ of $f$ in $\cC^{\mathrm{full}}$ obtained by taking inverse image of $\cV$ under the natural projection $\cC^{\mathrm{full}}\to \cC$. As  the $\cO(\cU)$-algebra $\cO(\cV^{\mathrm{full}})$  acts faithfully on   
the  $M^{\dag,\leqslant s}_\cU$, a projective  $\cO(\cU)$-module of finite rank, it follows that $\cO(\cV^{\mathrm{full}})$ is also projective as $\cO(\cU)$-module. It follows that $\bigcap_{i\in \Z} \gm_i \cO(\cV^{\mathrm{full}})=\{0\}$, where 
$(\gm_i)_{i\in \Z}$ is any Zariski dense set of maximal ideals of $\cO(\cU)$. Letting  the $\gm_i$'s 
correspond to classical weights $k_i$ which are large enough (with respect to the slope $s$), one deduces that  
$\cV^{\mathrm{full}}$ is reduced. To see that  any nilpotent element of $\cO(\cV^{\mathrm{full}})$ belongs to 
$\gm_i \cO(\cV^{\mathrm{full}})$ we recall that  $\cO(\cV^{\mathrm{full}})/\gm_i \cO(\cV^{\mathrm{full}})$ is a product of fields indexed by  $\cV\cap \sw^{-1}(k_i)$.  

It remains to show that the natural inclusion of reduced $\cO(\cU)$-algebras  $\cO(\cV)\to \cO(\cV^{\mathrm{full}})$ is an isomorphism, {\it i.e.} that $U_\ell\in \cO(\cV)$ for all $\ell$ dividing $N$.
One proceeds exactly as in \cite[Proposition~7.1]{bellaiche-dimitrov}   using $\rho_\cV$ 
to construct  an element of $\cO(\cV)$ whose value at each $g\in \cV\cap \cC^{\mathrm{cl}}$ is given by 
$U_\ell(g)$. 
\end{proof}

In view of Proposition~\ref{bad-Hecke-ops}, the irreducible components of $\cC$ containing a $p$-ordinary newform $f$ of tame level $N$, are in bijection with the minimal primes of the $p$-ordinary  Hida Hecke algebra of tame level $N$  which are contained in the 
height one prime attached to $f$ (see \cite{dim-durham} for more details).

\begin{rem} Analogues of Proposition \ref{bad-Hecke-ops} have  also been studied for $f\in \cC^{\mathrm{cl}}$ 
which are not cuspidal, but still are cuspidal-overconvergent, {\it i.e.} belong to the cuspidal eigencurve $\cC^{\cusp}$ defined in 
\S\ref{sec:hida-duality}. If $f$ has weight $1$ then it corresponds to a $p$-irregular  Eisenstein points and an exact analogue 
is proven in \cite[Proposition~4.4]{BDPozzi}. If $f$ has weight $k \geqslant 2$ then corresponds to a critical stabilization of an 
Eisenstein  series, and one can argue similarly using the Galois representation $\rho_\cV$ constructed by  
J.~Bella\"iche and G.~Chenevier \cite{bellaiche-chenevier}. The only potential problem is when 
the family  $\cV$ is generically Steinberg at $\ell$, while  $\rho_f$  is unramified at $\ell$, and the  study of the Weil-Deligne representation at $\ell$ then implies that $k=2$ (Eisenstein series do not satisfy the Ramanujan conjecture!).
This case was studied in detail in the PhD thesis of D. Majumdar \cite{majumdar}. 
A case which is  particularly piquant is that of the unique  weight $2$  level $\Gamma_0(\ell)$ 
Eisenstein series, which is {\it not} old at $\ell$, despite of $\rho_f$  being unramified at $\ell$, as the 
weight $2$ level  $1$ Einsenstein series  is not classical. 
\end{rem}

\subsection{Cuspidal Hida duality} \label{sec:hida-duality}
The    $\cO(\cU)$-submodule of $M^{\dag,\leqslant s}_\cU$  of cuspidal Coleman families   is defined as 
\[S^{\dag,\leqslant s}_\cU= \varinjlim_{v>0} \rH^0(X(v) \times \cU,\omega_\cU(-D))^{\leqslant s}, \] 
where  $D$ is the cuspidal divisor of  the ordinary locus $X^{\ord}$. 
Note that  it is $\cT_\cU^{\leqslant s}$-stable and one defines  
the cuspidal Hecke $\cO(\cU)$-algebra $\cT_\cU^{\cusp,\leqslant s}$ as the quotient of $\cT_\cU^{\leqslant s}$ acting faithfully on it. 
Using $\cT_\cU^{\cusp,\leqslant s}$ instead of $\cT_\cU^{\leqslant s}$  one defines the cuspidal eigencurve 
$\cC^{\cusp}$ which is endowed with a closed immersion $\cC^{\cusp}\hookrightarrow \cC$ of  reduced flat rigid curves over $\cW$. 

Let $f$ be a classical cuspidal point in $\cC$. As  recalled in \S\ref{ocmf},   $\sw(f)$ has a neighborhood  $\cU$ such that 
 $S^{\dag,\leqslant s}_\cU$ and $\cT_\cU^{\cusp,\leqslant s}$ are both free of finite rank as $\cO(\cU)$-modules. 
\begin{prop}\label{hidaduality}
Hida's  $\cO(\cU)$-linear pairing 
\[\langle \cdot , \cdot  \rangle: \cT_\cU^{\cusp,\leqslant s} \times S^{\dag,\leqslant s}_\cU\to \cO(\cU)\]
 sending $(T,\cG)\in \cT_\cU^{\cusp,\leqslant s} \times S^{\dag,\leqslant s}_\cU$  to    $\langle T, \cG \rangle= a_1(T(\cG))$ is  a perfect duality.
\end{prop}

\begin{proof} By Proposition \ref{bad-Hecke-ops} and the abstract recurrence relations between Hecke operators,
one know that for all $n\in \Z_{\geqslant 1}$ one has $T_n\in  \cT_\cU^{\cusp,\leqslant s}$. 
As 
\[\langle T, T_n(\cG) \rangle=\langle T_n T, \cG \rangle=a_1(T_n T(\cG))=a_n(T(\cG))\] 
for any $n \geqslant 1$, the $q$-expansion Principle from Proposition~\ref{q-expprinciple} shows that  the following natural $\cO(\cU)$-linear maps are  injective
\begin{align} \label{duality}
S^{\dag,\leqslant s}_\cU \longrightarrow  \Hom_{\cO(\cU)\text{-mod}}(\cT_\cU^{\cusp,\leqslant s},\cO(\cU)) \text{ , } & \cG \mapsto (T \mapsto \langle T, \cG \rangle)\\
  \cT_\cU^{\cusp,\leqslant s} \longrightarrow  \Hom_{\cO(\cU)\text{-mod}}(S^{\dag,\leqslant s}_\cU,\cO(\cU))  \text{ , }  &T \mapsto (\cG  \mapsto \langle T, \cG \rangle). 
\end{align} 
 In particular, if $\cU=\{\sw(f) \}$ then the above maps  are isomorphisms, as  $S^{\dag,\leqslant s}_{\sw(f)}$ and $\cT_{\sw(f)}^{\cusp, \leqslant s}$  are vector spaces of  finite dimension. 
For general $\cU$, as  $S^{\dag,\leqslant s}_\cU$ and $\cT_\cU^{\cusp,\leqslant s}$ are  free $\cO(\cU)$-modules of finite rank, 
 we proceed by localization at  $\sw(f)$. By Nakayama's lemma,  it suffices to  show that the cokernel of \eqref{duality} vanishes residually, which follows from the surjectivity in the particular case $\cU=\{\sw(f) \}$ that we already considered. 
\end{proof}

\begin{rem}
Similar techniques were used in \cite{BDPozzi} to establish a perfect duality between $M_\cU^{\dag,\leqslant s}$ and 
$\cT_\cU^{\leqslant s}$ locally at a point corresponding to a weight $1$  Eisenstein series irregular at $p$, the issue being the control of the constant terms. To that effect,   one  introduces the notion of  evaluation for Hida families at all cusps of $X^{\ord}$
and one  establishes a ``fundamental exact sequence''. 
\end{rem}

\section{Generalized eigenforms at classical points of the eigencurve}\label{geomhist}

\subsection{Geometry of  $\cC$ at classical points  of weight at least $2$} \label{weight-geq-2}
Recast in rigid geometry, Hida's famous Control Theorem \cite{Hida86} states that any  point of $\cC^{\ord}$ having classical weight $k\geqslant 2$ is classical. 
Note that ordinary forms of weight with $k\geqslant 2$ are always regular at $p$ and, when a second $p$-stabilization exists,  it  necessarily has critical slope. 
Coleman's Control Theorem generalizing  Hida's result,   states that any  non-critical point of $\cC$ having weight in  $\cW^{\mathrm{cl}}$ is classical
(see \cite{coleman-ocmf}). Those results imply that $\cC$ is \'{e}tale over  $\cW$ at any non-critical, $p$-regular classical point of weight $k\geqslant 2$, hence it is also smooth at these points. There are no known examples of classical forms of weight $k\geqslant 2$ which are $p$-irregular,  {\it i.e.},  for which the Hecke polynomial  at $p$  has a double root.
There are only three cases in which $\sw$ could fail to be \'{e}tale at a classical  point $f$ of weight $k\geqslant 2$ and regular at $p$, 
 potentially providing a cuspidal-overconvergent generalized eigenform.  

\begin{enumerate}
\item {\bf Critical  (a.k.a. evil) Eisenstein points.} Bella\"iche and Chenevier   proved in \cite{bellaiche-chenevier} that $\cC$ is smooth
and is conjecturally \'{e}tale over $\cW$ at such points. 

\item {\bf Critical  CM points.}  Bella\"iche showed  in  \cite{BelCM} that the eigencurve is smooth, although ramified over the weight space, at such points. He further showed that Jannsen's conjecture in Galois cohomology implies that the ramification degree  is precisely $2$. 

\item {\bf Critical  non-CM points.} Breuil and Emerton proved in  \cite{breuil-emerton} (see also \cite{ghate} for a partial 
result and \cite{bergdall} for a different proof) that $\sw$ ramifies at a classical weight $k\geqslant 2$ point of critical slope 
if, and only if, there exists  an ordinary companion form, in which case the restriction to $\G_{\Q_p}$ of the attached $p$-adic Galois representation splits. It is a folklore conjecture,  attributed to R.~Coleman and to R.~Greenberg, that such non-CM points should not exist (see for example \cite{castella-wang}). 
\end{enumerate}

Note that if a $p$-irregular classical weight $k\geqslant 2$  points were to exist, they would be of non-critical  slope 
and Coleman's Classicality Theorem \cite{coleman-ocmf} would imply that the corresponding generalized eigenspace  consists only  of classical 
forms. Thus the second case above is (conjecturally) the only one yielding cuspidal-overconvergent generalized eigenforms, and their $q$-expansions have recently been  computed by Hsu  \cite{hsu} (note that the technical condition preceding Theorem 1.1 in {\it loc. cit.} appears to be superfluous). 

\subsection{Geometry of  $\cC$ at classical points  of weight $1$ and Hida theory}
Classical  weight $1$ points in $\cC$ all belong to $\cC^{\ord}$ and all have critical slope, hence Hida's  Control  Theorem cannot be applied them and  a specializations of a Hida family in weight $1$  need not a  classical eigenform. 
According to a result of E.~Ghate and V.~Vatsal \cite{ghate-vatsal}, only  Hida families with complex multiplication (CM) 
admit infinitely many classical weight $1$  specializations. An explicit bound for the number of classical weight $1$  specializations of a non-CM family can be found in   \cite{dim-ghate}. 

 Deligne and Serre attached in \cite[Proposition~4.1]{deligne-serre} to any  weight $1$ newform  $f$  an irreducible Galois representation  $\rho_f:\G_{\Q} \rightarrow  \GL_{2} (\bar\Q)\xrightarrow{\iota_p} \GL_{2} (\bar\Q_p)$ having 
finite image. As well-known (see for example \cite{dim-ghate}),  the projective image is either dihedral,  or else is isomorphic to one of the groups $A_4$, $S_4$ or $A_5$ in which case the form (or the corresponding point on $\cC$) is referred to as exotic. 
When the projective image is dihedral, then $f$ has real multiplication (RM) or 
complex multiplication (CM),  depending on whether the corresponding quadratic extension of $\Q$  is real or imaginary.  
Note that, when the projective image is the Klein four group, then $f$ has multiplication by three quadratic fields (two imaginary and one real). 

One motivation for studying the geometry of the eigencurve at weight $1$ points arises from the question of determining whether there is a unique, up to Galois conjugacy,  Hida family  specializing to a given weight $1$ eigenform $f$. By
Proposition~\ref{bad-Hecke-ops} the 
minimal primes in Hida's $p$-ordinary Hecke  algebra which are contained in the height $1$ prime attached to $f$  are in bijection with the irreducible components of $\cC$ containing $f$. It follows that if   $\cC$ is smooth at $f$, then there exists a   unique, up to Galois conjugacy,  Hida family specializing to  $f$. 
It is observed in \cite{dim-ghate} that any classical weight $1$ form  of Klein type and irregular at $p$ is contained in two Hida families having CM by different imaginary quadratic fields, for which reason they cannot be  Galois conjugates. 
It turns out that this is part of a more general phenomenon. Indeed, it is proven in 
\cite{betina-dimitrov} that any  weight $1$  form irregular at $p$ and having CM by $K$ belongs to exactly  
$3$ or $4$ irreducible components of $\cC$, exactly $2$ out of which have CM by $K$. 
In particular the uniqueness, up to Galois conjugacy,  of the Hida family systematically fails at such weight $1$ classical points.

While conjecturally  all classical points of weight $k\geqslant 2$ are  expected to be $p$-regular (see \cite{coleman-edihoven}), the Chebotarev Density theorem applied to  
$\rho_f$, shows that for each classical   weight $1$ point $f$ there are infinitely many irregular primes $p$, providing many classical points at which $\sw$ is not \'{e}tale (and not even smooth if $f$ has CM).  

The main result of \cite{bellaiche-dimitrov} asserts that $\cC$ is smooth at any classical weight $1$ point $f$  which is regular  at $p$. 
Furthermore,  the weight map $\sw$ is not \'{e}tale at $f$ if, and only if, $f$ has real multiplication  (RM) by a  quadratic field   in which $p$ splits. The next subsection is devoted to this case. 

\subsection{Regular RM case}\label{RM-case} 

Let $f$ be a $p$-stabilization of a $p$-regular weight $1$ newform of level $N$ having  multiplication by a real quadratic  field in which $p$ splits, 
and let $\gm$  denote the corresponding maximal ideal of  the  Hecke algebra. 
In this case the classical subspace of the generalized eigenspace $S_{\sw(f)}^\dagger\lsem f \rsem\subset 
S_{\sw(f)}^{\dagger}$ is given by the line 
$\bar\Q_p\cdot f=S_1(Np)[\gm]$. Moreover a natural supplement of $\bar\Q_p\cdot f$ in $S_{\sw(f)}^\dagger[\gm^2]$
is given by the subspace  $S_{\sw(f)}^\dagger[\gm^2]_0$ of cuspforms whose  first Fourier coefficient  vanishes ($a_1=0$). 
As $S_{\sw(f)}^\dagger[\gm^2]_0$ is naturally isomorphic to the relative tangent space of $\cC$ over $\cW$ at $f$
 ({\it i.e}, the tangent space of the fiber of $\sw^{-1}(\sw(f))$ at $f$), the results of \cite{bellaiche-dimitrov} show that $S_{\sw(f)}^\dagger[\gm^2]_0$
 is a line, having a  basis $f^\dagger$. Using a  cohomological computations from \cite{bellaiche-dimitrov}, 
  H.~Darmon, A.~Lauder and V.~Rotger determine in \cite{DLR-Adv}   the precise $q$-expansion of this genuine overconvergent generalized eigenform $f^\dagger$. They also draw some parallels with the famous Hilbert's twelfth problem which remains unsolved for   real quadratic fields.

We use the present opportunity to observe that, contrarily to what was claimed in \cite{bellaiche-dimitrov},  based on a  
misinterpretation of a result by Cho and Vatsal \cite{cho-vatsal},  the expected equality $S_{\sw(f)}^\dagger\lsem f\rsem=S_{\sw(f)}^\dagger[\gm^2]$ remains an  open question.  It is equivalent to showing that the ramification index of $\sw$ at $f$ equals $2$ (see \cite{betina-CJM} for a thorough study of this case).

The geometry of the eigencurve is expected to  be more intricate at a weight $1$  points irregular at $p$ and to have
fascinating   applications in Iwasawa theory. 

\subsection{Eisenstein case} \label{eis-sec}
We refer to  \cite{BDPozzi}  for a detailed study of this case. 
A $p$-regular weight $1$ Eisenstein point belongs to a unique irreducible component of $\cC$ which is 
Eisenstein and \'{e}tale over the weight space. A $p$-irregular weight $1$ Eisenstein point $f$ is 
 attached to an odd primitive Dirichlet character $\phi:(\Z/N\Z)^\times\to \bar\Q_p^\times$ such that $\phi(p)=1$. 
 In contrast with the $p$-regular case, $f$  belongs to exactly two Eisenstein components. Moreover $f$  is cuspidal-overconvergent, {\it i.e.}
vanishes at all cusps of the multiplicative ordinary locus of the modular curve $X(\Gamma_0(p) \cap \Gamma_1(N))$ corresponding  to the  $\Gamma_0(p)$-orbit of $\infty$, hence $f$  belongs to the cuspidal eigencurve $\cC^{\cusp}$. It   is 
shown in {\it loc.cit.} that  $\cC^{\cusp}$ is  \'{e}tale  at $f$ over the weight space,  hence there exists a unique, up to a Galois conjugacy, cuspidal Hida family $\cF$ specializing to $f$. 
 It is also shown that $\cC^{\cusp}$ intersects transversally each of the two Eisenstein components containing $f$ 
(note that the evil weight Eisenstein series of weight $\geqslant 2$ do not belong to Eisenstein components). 
Furthermore, one computes in  {\it loc.cit.} the  $q$-expansions of a basis $\{f^\dagger_{1,\phi}, f^\dagger_{\phi,1}\}$ 
of the space of genuine overconvergent generalized eigenforms  in terms of $p$-adic logarithms of algebraic numbers, 
and one remarks that these forms are {\it  not} cuspidal-overconvergent. Finally, the expression of the constant term 
 $a_0(f^{\dagger}_{1,\phi})=(\cL(\phi) +\cL(\phi^{-1})) \frac{L(\phi,0)}{2}$, where $\cL(\phi)$ denotes the  cyclotomic $\cL$-invariant 
 appearing in the derivative at a   trivial zero of the Kubota--Leopoldt $p$-adic $L$-function $L_p(\phi\omega_p,s)$, allows to give a geometric flavored   proof of Gross' formula $L_p'(\phi\omega_p,0)=\cL(\phi)\cdot L(\phi, 0)$ (see \cite[\S5]{BDPozzi}).

\subsection{The $p$-irregular weight one case}\label{sec-irregular}
Let us first say that this case  is still   the  subject of active  research. 
A weight one  newform  which is irregular at $p$ yields a unique  point $f$ on~$\cC$. 

Let  $\cT$ be the completed local ring of $\cC$ at $f$ and  $S^{\dag}_{\gm}$ be the $\cT$-module  obtained by localizing and completing
 $S^{\dagger, \leqslant s}_\cU$ at the maximal ideal $\gm$ of  $\cT_{\cU}^{\leqslant s}$ corresponding to the system of Hecke eigenvalues of $f$, where $\cU$ is an admissible  affinoid of $\cW$ containing  $\sw(f)$.  The localization at $\gm$  of the pairing defined in  Proposition  \ref{hidaduality} gives rises to a perfect pairing 
 \begin{equation} \cT \times S^{\dag}_\gm\to\Lambda , \text{ given by } \langle h, \cG \rangle= a_1(T\cdot \cG) \in \Lambda,  \end{equation}
 where $\Lambda$ is the completed local ring of $\cW$ at $\sw(f)$. Specializing in weight $\sw(f)$, corresponding to the maximal ideal 
 $(X)$ of $\Lambda$, yields  a natural isomorphism
 \begin{equation}
 S_{\sw(f)}^{\dag}\lsem f \rsem=S^{\dag}_{\gm}/\gm_{\Lambda}S^{\dag}_{\gm} \xrightarrow{\sim} \Hom_{\bar\Q_p}\left(\cT/ (X\cdot \cT), \bar\Q_p\right). 
 \end{equation}
Since $\cT/ X\cdot \cT$ is an Artinian  $\bar\Q_p$-algebra, the space  
$S_{\sw(f)}^{\dag}\lsem f \rsem$ of overconvergent
weight $\sw(f)$ generalized eigenforms is by definition the union over all 
$i \geqslant 1$ of its  subspaces $S_{\sw(f)}^{\dag}[\gm^i]$ annihilated by the ideal $\gm^i$.

Let us observe that the classical subspace
of the generalized eigenspace $S_{\sw(f)}^\dagger\lsem f\rsem$ is given by the plane 
$\bar\Q_p\cdot f(z)\oplus \bar\Q_p\cdot  f(pz) =S_1(\Gamma_0(p) \cap \Gamma_1(N))[\gm^2]$ which 
has a natural supplement $S_{\sw(f)}^\dagger[\gm^2]_0$ in $S_{\sw(f)}^\dagger[\gm^2]$, consisting of cuspforms whose  first and $p$-th Fourier coefficients both vanish ($a_1=a_p=0$). As in \S\ref{RM-case} one can reasonably conjecture that $S_{\sw(f)}^\dagger\lsem f\rsem=S_{\sw(f)}^\dagger[\gm^2]$.

In \cite{DLR4}   Darmon, Lauder and Rotger constructed a  map  
\begin{equation}\label{DLR-map}
S_{\sw(f)}^\dagger[\gm^2]_0 \to \rH^1(\Q,\ad^0\rho)
\end{equation}
and conjectured that it is an isomorphism. A first evidence was found by Hao Lee \cite{lee} when $f$ is of  Klein type. The only other case where a full study of the local geometry has been successfully completed is the case of a $p$-irregular
weight $1$ CM form, presented in the following subsection.

\subsection{The $p$-irregular CM case}\label{CM-case}
We will use without recalling the notations and assumptions introduced in the paragraph preceding Theorem \ref{main-q-exp}.

Since $\psi\ne \bar\psi$, the reducibility $\rho_{f\mid \G_K}= \psi\oplus \bar\psi$ allows to choose a basis of eigenvectors $(e_1,e_2)$ which is uniquely defined 
 up to individual scaling.  
 Using the complex conjugation $\tau$ to further impose that $e_2 = \rho(\tau) e_1$  determines projectively uniquely this basis
 and one has: 
 \begin{equation}\label{rho}
\rho_{f\mid \G_K} = \left(\begin{matrix} \psi & 0 \\ 0 & \bar\psi \end{matrix}\right) 
\qquad \rho_{f \mid \G_\Q \setminus \G_K} = \left(\begin{matrix} 0 & \psi(\cdot \tau) \\ \psi(\tau \cdot) & 0 \end{matrix}\right).
\end{equation}

The local-global compatibility for $\rho_f=\Ind^\Q_K \psi$ yields that 
 \begin{equation}\label{qexp-f}
 a_\ell=\begin{cases} \psi(\gl)+\psi(\bar\gl) &,  \text{ if } (\ell)=\gl\bar\gl \ne (p) \text{ splits in  } K, \\
 0 &,  \text{ if } \ell \text{ is inert in } K, \\
\psi(\gl) &,  \text{ if } \gl\mid \ell\mid p D, \end{cases}
\end{equation}
which together with the usual recurrence relations $a_{\ell^{r+1}}=a_\ell\cdot a_{\ell^r}-\varepsilon(\ell)a_{\ell^{r-1}}$, for 
$r\in\Z_{\geqslant 1}$, and $a_{m n}= a_m a_n$, for $m, n$ relatively prime, uniquely determine  $f$. 
The above formulas are understood with the convention that $\psi(\gl)=0$, if $\gl$ divides the conductor of $\psi$, and similarly  
$\varepsilon(\ell)=0$, if $\ell$ divides the conductor of $\rho_f$. 

It has been shown in \cite{betina-dimitrov} that in addition to belonging to two components of $\cC$ having CM by $K$, $f$ also belongs to   one or two other components,  {\it i.e.}, $\cO_{\cC,f}$, as well as the Hida Hecke algebra localized at $f$, 
 have exactly three or four minimal primes.   According to Hida's work in Iwasawa theory,  points lying at the intersection of  CM and non-CM families are expected to correspond to zeros of anti-cyclotomic Katz $p$-adic $L$-functions. 
In our situation, we are in the presence of a so-called ``trivial'' zero and, prior to  {\it loc.cit.}, 
 one ignored whether this zero was simple or not. Indeed,  there is no an ``anti-cyclotomic'' analogue of the famous Ferrero--Greenberg Theorem on the 
Kubota--Leopoldt $p$-adic $L$-functions of a Dirichlet character encountered in the Eisenstein case described in \S\ref{eis-sec}. 
By using  $p$-adic geometry, commutative algebra and  Galois theoretic tools together, it is shown in {\it loc.cit.} is that these
``anti-cyclotomic'' trivial zeros are simple whenever a certain $\cL$-invariant does not vanish, as predicted by the  Strong Four Exponentials Conjecture.

A corollary of the main results of \cite{betina-dimitrov} is a proof of the Darmon--Lauder--Rotger Conjecture on \eqref{DLR-map} being an isomorphism, 
for all $p$-irregular  weight $1$ CM forms.  In the next section we will compute the $q$-expansions of a basis $\{f^\dagger_{\cF}, f^\dagger_{\Theta}\}$  of the 
genuine generalized eigenspace $S^\dagger[\gm^2]_0$.

\section{Overconvergent \texorpdfstring{$q$}{q}-expansions at \texorpdfstring{$p$}{p}-irregular CM forms} \label{sec-proofs}

The purpose of this section is to prove Theorems \ref{main-q-exp}  and \ref{thm-cross-ratio} which are a natural  extension of the results of \cite{betina-dimitrov}. 
 We keep the setting and  notations  from \S\ref{CM-case} and we let  $\nu=2$ if $p=2$, and $\nu=1$ otherwise.  Define 
 \[ 
  \cL=\frac{\cLm(\varphi)}{\log_p(1+p^{\nu})\cdot(\cLm(\varphi)+\cLm(\bar\varphi))}, \,\,\,
 \bar\cL=\frac{\cLm(\bar\varphi)}{\log_p(1+p^{\nu})\cdot(\cLm(\varphi)+\cLm(\bar\varphi))}.  \]

 \subsection{Infinitesimal non-CM Hida families} 
 
Under the assumption \eqref{generic-CM}, the results of  \cite{betina-dimitrov}  recalled in \S\ref{CM-case}  imply that  $f$ belongs to 
 exactly four irreducible components of $\cC$, all \'{e}tale over $\Lambda$, two having CM by $K$, while the other two corresponding to a Hida 
family $\cF= \sum_{n \geqslant 1} a_n(\cF) q^n  \in \Lambda \lsem q \rsem $ without  CM by $K$, and to its quadratic twist 
$\cF \otimes \varepsilon_K$.

Consider the unique cochains $\eta_\varphi:\G_K\to \bar\Q_p$ representing the cocycles $[\eta_\varphi]\in \rH^1(K,\varphi)$, 
normalized so that $\res_{\gp}([\eta_\varphi])=\log_p\in \rH^1(\Q_p,\bar\Q_p)$ 
and $\eta_\varphi(\gamma_0)=0$ for a fixed element $\gamma_0\in \G_K$ such that $\varphi(\gamma_0)\ne 1$ (see  \cite[\S1]{betina-dimitrov}). 
The slope $\cS_{\varphi}$ is a non-zero $p$-adic number defined by $\res_{\bar\gp}([\eta_{\varphi}])=\cS_{\varphi}\cdot [\log_p]$. 
As explained in {\it loc. cit. }  one has the following  formula
\begin{equation}\label{slope-def}
\cS_{\varphi}=-\frac{\log_p(u_{\varphi})}{\log_p(\tau(u_{\varphi}))},
\end{equation}
where $\bar\Q \cdot u_{\varphi}= (\bar\Q\otimes\cO_H^{\times} )[\varphi]$, and moreover $\cS_{\bar\varphi}= \cS_{\varphi}^{-1}$.

\begin{prop}[{\cite[\S2]{betina-dimitrov}}] \label{prop-tangent}
Any $\G_{K/\Q}$-stable ordinary infinitesimal deformation of $\rho=\left(\begin{smallmatrix} \psi & \psi\eta_\varphi \\ 0 & \bar\psi \end{smallmatrix}\right)$ is reducible,  {\it i.e.},  of the form 
$\left(\begin{smallmatrix} 1+\bar d  X& b X \\ 0 & 1+d X  \end{smallmatrix}\right) \cdot\rho  \mod{X^2}$
and one has a natural isomorphism: 
\[ t_\rho^{\ord} \xrightarrow{\sim} \rH^1(K,\bar\Q_p), \,\,\, \left[\left(\begin{smallmatrix} \bar d & b \\ 0 & d   \end{smallmatrix}\right)\right] \mapsto d.\]
\end{prop}

Recall the basis $\{\eta_\gp, \eta_{\bar\gp}\}$ of $\rH^1(K,\bar\Q_p)$ where $\eta_\gp$, resp. $\eta_{\bar\gp}$, is unramified outside $\gp$, resp. 
$\bar\gp$, and $\res_{\gp}(\eta_\gp)=\res_{\bar\gp}(\eta_{\bar\gp})=\log_p$. One has $\eta_\gp+\eta_{\bar\gp}=\eta_{p\mid \G_K}$, where
$\eta_{p}\in \rH^1(\Q,\bar\Q_p)$ is such that $\res_{p}(\eta_{p})=\log_p$.
Writing $d=x\cdot \eta_{\bar\gp} + y\cdot \eta_\gp$ Figure~\ref{graph}  represents $t_\rho^{\ord}$ and its  relevant subspaces.

\begin{figure}[h!]
\centering
\begin{tikzpicture}
\draw[thick,] (0,0) -- (7,0) node[anchor=south east]  {{\small $x$-axis}};
\draw[thick,] (0,-2.25) -- (0,3) node[anchor=north west ] {{\small $y$-axis}};
\draw[thick,] (0,0) -- (7,3) node[anchor= north east] { \hspace{-5cm}  {\small   $t_\cF:x\cdot \cL=y\cdot \bar\cL$}};
\draw[thick,] (0,0) -- (-5,0) node[anchor=south west ] {{\small CM line}};
\draw[thick,] (-5,-2.10) -- (0,0)  node {$\bullet$}  node[anchor=north east] {$f$};
\draw[thick,] (-1,3) -- (4.25,-2.25) node[anchor=south west] {{\small   $x+y=-\log_p^{-1}(1+p^{\nu})$}};
\draw[thick]  (0,0) -- (2.25,-2.25) node[anchor=south west] {{\small   $x+y=0$}};
\draw[thick]  (0,0) -- (-3,3) node[anchor= north east] { {\small   weight $1$ line}};
\draw (1.4,0.6) node {$\bullet$} ;
\draw (1.5,0.7) node [anchor=south] {{\small $\cF$}};
\draw (2,0) node {$\bullet$} ;
\draw (2.1,0) node [anchor=south] {{\small $\Theta_\psi$}};
\draw (-0.6,0.6) node {$\bullet$} node [anchor=south] {{\small $f^{\dagger}_{\psi}$}};
\end{tikzpicture}
\caption{Ordinary tangent space of $\rho$}
\label{graph}
\end{figure}

We note that $t_\cF$ corresponds to the tangent space  of a  $\Lambda$-adic deformations of $\rho$ having no CM by $K$, 
whereas the line $x+y=-\log_p^{-1}(1+p^{\nu})$ corresponds to the  deformations whose determinant differs from 
$\det(\rho)$ by the infinitesimal  universal cyclotomic character $ \left(1-\frac{\eta_{p}}{\log_p(1+p^{\nu})}X \right)$
(see \cite[\S2.4]{BDPozzi}) 
We refer to  the closing  Remark~\ref{rem:f-dagger} for a definition of the genuine overconvergent generalized eigenform $f_\psi^{\dagger}$.

\begin{prop} \label{infinitesimal-rhoF}
There exists a basis where $\rho_\cF(\tau) \equiv  \left( \begin{smallmatrix} 0& 1 \\ 1 & 0  \end{smallmatrix}\right) \mod{X^2}$, and 
\[\rho_{\cF\mid \G_K}\equiv
 \left(\left(\begin{matrix} 1 & 0
 \\ 0  & 1 \end{matrix}\right) - \left(\begin{matrix} \bar\cL\eta_\gp+\cL\eta_{\bar\gp} &  \xi \cL \eta_{\varphi} 
 \\    \xi^{-1} \bar\cL \cS_{\bar\varphi} \bar\eta_{\varphi}  &  \cL\eta_\gp+\bar\cL\eta_{\bar\gp}  \end{matrix}\right)\cdot X \right)
 \left(\begin{matrix} \psi &  \\ & \bar \psi \end{matrix}\right)\mod{X^2}, \]
 where  $\xi\in\bar\Q_p$ is such that  $\xi^2 =\cS_{\bar\varphi} \frac{\bar\cL}{\cL}$. 
 Moreover, the $\G_{K_\gp}$-stable line of $\rho_{\cF}$ is generated residually by 
 $\xi\cdot e_1+ e_2$ and the corresponding  unramified quotient  is given by the character 
 \[\chi_{\cF}\equiv \psi\cdot\left(1- ( \cL(\eta_\gp-\eta_{\varphi})+ \bar\cL\eta_{\bar\gp})\cdot X\right) \mod{X^2}. \]
 
   The expression $\rho_{\cF \otimes \varepsilon_K}$ is obtained by replacing $\xi$  with $-\xi$. 
 \end{prop}
 
 \begin{proof} The claim about $\rho_{\cF\mid \G_K}$ follows from Proposition~\ref{prop-tangent}, Figure~\ref{graph}, the equality $(\xi \cL)^2=\cS_{\bar\varphi} \bar\cL \cL$ and  the fact,  proved in \cite[Lemma~3.10]{betina-dimitrov}, that its reducibility ideal equals $(X^2)$. Note that while the cocycles $\cS_{\bar\varphi} \bar\eta_{\varphi}$  and   $\eta_{\bar\varphi}$ from {\it ibid.} might differ by a coboundary, this coboundary necessarily vanishes on $\G_{K_\gp}$,  facilitating  the computation of  the ordinary filtration.

  It follows from {\it loc. cit.} that 
 $\rho_\cF(\tau)\equiv \left(\begin{smallmatrix} 0  & \mu^{-1} \\ \mu & 0 \end{smallmatrix}\right)\mod{X^2}$, for some 
 $\mu \in \bar\Q_p^\times$. By rescaling the basis by an element of $1+\bar\Q_pX$ 
  (which would not alter the expression of $\rho_{\cF\mid \G_K}$ modulo $X^2$), one can find  $\mu' \in \bar\Q_p$ such that 
  \[\rho_\cF(\tau)\equiv   \left(\begin{smallmatrix} -\mu' X  & \mu^{-1} \\ \mu & \mu' X \end{smallmatrix}\right)
   \mod{X^2}.\]
 
   Computing $\rho_{\cF\mid \G_K}(\tau\cdot\tau)$ one finds that: 
 \begin{align*} \left(\begin{smallmatrix}  \cL\eta_\gp+\bar\cL\eta_{\bar\gp} &\xi \cL\bar\eta_{\varphi}
 \\   \xi^{-1} \bar\cL \cS_{\bar\varphi} \eta_{\varphi}  & \bar\cL\eta_\gp+\cL\eta_{\bar\gp}   \end{smallmatrix}\right)
 \left(\begin{smallmatrix} 1 &  \\ & \varphi \end{smallmatrix}\right)
 = &\\
 \left( \begin{smallmatrix} 0& \mu^{-1} \\ \mu & 0  \end{smallmatrix}\right)
  \left(\begin{smallmatrix} \bar\cL\eta_\gp+\cL\eta_{\bar\gp} & \xi \cL\eta_{\varphi} 
 \\   \xi^{-1} \bar\cL \cS_{\bar\varphi} \bar\eta_{\varphi}  &  \cL\eta_\gp+\bar\cL\eta_{\bar\gp}  \end{smallmatrix}\right)&
 \left(\begin{smallmatrix} \varphi &  \\ & 1 \end{smallmatrix}\right)
 \left( \begin{smallmatrix} 0& \mu^{-1} \\ \mu & 0  \end{smallmatrix}\right)+
  \left( \begin{smallmatrix} 0& \mu^{-1}\mu'(\varphi-1) \\ \mu\mu'(\varphi-1) & 0  \end{smallmatrix}\right). 
 \end{align*}
  
 Since $\varphi$ is non-trivial  the above equality implies   that $\mu'=0$ and $\mu^2=1$. 
 If $\mu=-1$,   then one can change the signs of $\mu$ and $\xi$ simultaneously, to obtain 
  $\rho_\cF(\tau) \equiv  \left( \begin{smallmatrix} 0& 1 \\ 1 & 0  \end{smallmatrix}\right) \mod{X^2}$ in all cases. 
\end{proof} 

\subsection{Some $\ell$-units} \label{l-units}

Recall the $\cL$-invariant $\cL_{\gp}=-\frac{\log_p (u_\gp)}{\ord_{\gp}(u_\gp)}$ defined in \cite{betina-dimitrov}, where
 $u_\gp \in \cO_{K}[\frac{1}{\gp}]^\times $ is any element having non-zero  $\gp$-adic valuation. 
Analogously given a  prime $\gl$ of $K$, relatively prime to $p$,  we let 
\begin{equation}\label{defn-ellunit}
\cL_{\gl}=-\frac{\log_p (u_{\gl})}{\ord_{\gl}(u_{\gl})},
\end{equation}
where  $u_\gl \in \cO_{K}[\tfrac{1}{\gl}]^\times $ is any element whose $\gl$-adic valuation $\ord_{\gl}(u_{\gl})$ is non-zero. 
Clearly, $\cL_{\gl}$ only depends on $\gl$ (and not on the  choice of  $u_\gl$), and it 
equals $-\log_p(\ell)$ (resp. $-\tfrac{1}{2}\log_p(\ell)$) when  
 $\ell$ is inert (resp. ramified) in $K$. 
 
 We let $H$ denote the splitting field of the anti-cyclotomic character $\varphi$.

 Assume for the rest of \S\ref{l-units} that   $\ell$ is inert or ramified in $K$,  and that 
 $\ell D\nmid N$. Then $\psi$ is unramified at the unique prime $\gl$  of $K$ above $\ell$ and 
\begin{equation}\label{trivial-at-ell}
\varphi(\gl)=\psi(\gl)/\psi(\bar\gl)=1,
\end{equation}
{\it i.e.} $\gl$ splits completely in $H$.  
Let  $\lambda$ be a prime of $H$ above $\gl$ and choose an element  $u_\lambda\in  \cO_H[\tfrac{1}{\lambda}]^{\times}$ whose $\lambda$-adic valuation  $\ord_\lambda(u_\lambda)$ is non-zero. 
Using \eqref{trivial-at-ell} one can prove that 
 $(\bar\Q \otimes \cO_H[\tfrac{1}{\ell}]^\times)[\varphi]$ is a $\bar\Q$-plane with a basis consisting of $u_{\varphi}$ and 
\begin{equation}\label{unit lambda} 
u_{\varphi,\lambda} = \sum_{h \in \Gal(H/K)} \varphi^{-1}(h) \otimes  h(u_\lambda) \in (\bar\Q \otimes   \cO_H[\tfrac{1}{\ell}]^\times)[\varphi].
\end{equation}

 By \eqref{slope-def} the following  does not depend on the particular choice of  $u_\lambda$\begin{equation}\label{inert-L-inv}
\cL_{\varphi,\lambda}=-\frac{\log_p(u_{\varphi,\lambda}) + \cS_{\varphi}   \log_p(\tau(u_{\varphi,\lambda}))}{\ord_\lambda(u_\lambda)}. 
\end{equation}

Let us now investigate  how $\cL_{\varphi,\lambda}$  depends  on $\lambda$.   Since 
$\Gal(H/K)$ acts transitively on the primes $\lambda$ of $H$ above $\ell$, any such prime is of the form  $h(\lambda)$ for some $h\in \G_K$. As $h(u_{\lambda})=u_{h(\lambda)}$, one has  $u_{\varphi,h(\lambda)}=\varphi(h)\cdot u_{\varphi,\lambda}$ and 
\[\cL_{\varphi,h(\lambda)}=\varphi(h)\cdot \cL_{\varphi,\lambda}. \]
 On the other hand, given   any element   $\gamma\in \G_\Q\backslash \G_K$, one has
\[
 \psi(\tau h\gamma h^{-1}) =
\psi(\tau h\tau\cdot \tau\gamma\cdot h^{-1})=\varphi(h)^{-1}\psi(\tau \gamma),
\]
hence  $\psi(\tau\cdot  h \gamma h^{-1}) \cdot  \cL_{\varphi,h(\lambda)}=
 \psi(\tau \gamma) \cdot  \cL_{\varphi,\lambda}$.  
 
 We will now make the following specific choice for $\gamma$: 
 
$\bullet$  if $\ell$ is inert in $K$, we let  $\gamma$  be  any Frobenius element $\Frob_\ell\in \G_\Q$ whose image in $\Gal(H/\Q)$ equals the Frobenius at $\lambda$; 
 
$\bullet$  if $\ell$ is ramified in $K$, we let  $\gamma$  be any  element of an inertia subgroup of $\G_\Q$ at $\ell$  whose image in $\Gal(H/\Q)$ generates the order two subgroup $\Gal(H_\lambda/\Q_\ell)=
\I(H_\lambda/\Q_\ell)$. 

\begin{defn}\label{ell-L-inv} Let $\ell$ be a prime which is inert or ramified in $K$ and such that $\ell D\nmid N$. For $\gamma$ chosen as above we define  
\begin{equation}\label{psi-invariant}
\cL_{\psi,\ell} =   \psi(\tau \gamma) \cdot  \cL_{\varphi,\lambda} \in \bar\Q_p.
\end{equation}
\end{defn}
As the restriction of $\gamma$ to $\bar\Q^{\ker(\rho_f)}$
is uniquely determined, up to $\G_H$-conjugacy,   it follows that $\cL_{\psi,\ell}$ only depends on $\ell$ and on $\psi$, and not on the particular  choices of $\lambda$ and of $\gamma$  as  above.

\subsection{Computation at split primes}
Assume that the prime $\ell$ splits in  $K$, and  write $(\ell)=\gl\bar{\gl}$ with $\gl\ne\bar{\gl}$. 
We recall that  $a_\ell=\psi(\gl)+\psi(\bar\gl)$, for $\ell\ne p$ and $a_p=\psi(\gp)=\psi(\bar\gp)$. 
We will now determine $a_\ell(\cF)$ infinitesimally.

\begin{prop} \label{prop-split} Let  $\ell\ne p$  be a prime  splitting in $K$ as  $\gl \cdot \bar{\gl}$.
  
One has  $\eta_{\gp}(\Frob_\gl)=\cL_{\gl}$, $\eta_{\bar\gp}(\Frob_\gl)=\cL_{\bar\gl}$ and 
\[
\left.\tfrac{d}{dX}\right\mid _{X=0}a_\ell(\cF)=
- \psi(\gl)\cdot (\bar\cL  \cL_{\gl}+ \cL \cL_{\bar\gl})- \psi(\bar\gl)\cdot (\bar\cL  \cL_{\bar\gl}+ \cL\cL_{\gl}).
\]

 Moreover $\left.\tfrac{d}{dX}\right\mid _{X=0}a_p(\cF)= \frac{\psi(\gp)}{\log_p(1+p^{\nu})} \cdot \left(\frac{ \cLm(\varphi) \cdot \cLm(\bar\varphi)}{\cLm(\varphi)+\cLm(\bar\varphi)}+\cL_{\gp}\right)$. 
\end{prop}

\begin{proof} By Class Field Theory one has an exact sequence
\[0\to \Hom(\G_K,\bar\Q_p)\to \Hom(\cO_{K,\gp}^\times \times \cO_{K,\bar\gp}^\times \times K_{\gl}^\times ,\bar\Q_p)
\to \Hom(\cO_{K}[\tfrac{1}{\gl}]^\times,\bar\Q_p)\]
whose first map sends $\eta_\gp$ to $(\log_p,0,\eta_{\bar{\gp}}(\Frob_\gl)\cdot \ord_{\gl})$. Hence 
$\log_p (u_{\gl})+ \eta_{\gp}(\Frob_\gl)\ord_{\gl}(u_{\gl})=0$ as claimed (see \eqref{defn-ellunit}). 
Similarly  $\eta_{\bar{\gp}}(\Frob_\gl)=\cL_{\bar{\gl}}$. 

  One can associate  to $\cF$  a $p$-ordinary  $\Lambda$-adic representation 
 $\rho_{\cF}:\G_\Q \to \GL_2(\Lambda)$ whose trace is given by the pushforward of $\tau_{\cC}$ introduced in  \eqref{pseudo-char}.  
 It follows that   for any $\ell \ne p$ 
\[ a_\ell(\cF)=\tr \rho_{\cF}^{\I_\ell}(\Frob_\ell), \text{ and } a_p(\cF)=\chi_{\cF}(\Frob_p),\]
where $\chi_{\cF}$ is the unramified character acting on the unramified $\G_{K_{\gp}}$-quotient of $\rho_{\cF}$.
The computation of $\left.\tfrac{d}{dX}\right\mid _{X=0}a_\ell(\cF) $ then follows directly from the infinitesimal expression for
$\rho_{\cF\mid \G_K}$ given in Proposition \ref{infinitesimal-rhoF}. The value  
 $\left.\tfrac{d}{dX}\right\mid _{X=0}a_p(\cF)$ is  computed similarly  using the formulas $(\eta_{\varphi}-\eta_\gp)(\Frob_\gp)=\cLm(\bar\varphi)+\cL_{\gp}$ and  $\eta_{\bar\gp}(\Frob_\gp)=-\cL_{\gp}$ established in \cite{betina-dimitrov}. 
 \end{proof}

\subsection{Computation at  inert primes}
We now turn to the case of a prime  $\ell$ which is  inert in $K$. If $\ell\nmid N$, then we let  $\gamma=\Frob_\ell\in\G_\Q\backslash \G_K$ 
denote    a Frobenius element and let  $\lambda$ denote the resulting place of $H$ above $\ell$, 
{\it i.e.} $\Frob_\ell^2\in \G_H$ defines a Frobenius element at $\lambda$, denoted $\Frob_{\lambda}$. 
 We have seen that $a_\ell=0$ and the aim of this section is to express infinitesimally $a_\ell(\cF)$ in terms of logarithms of  $\ell$-units.

\begin{prop}\label{prop-inert}
Assume that $\ell$ is inert in $K$. If $\ell \mid N$, then $a_\ell(\cF)=0$.  

If $\ell \nmid N$  then 
\begin{enumerate}
\item  $\eta_{\varphi}(\Frob_{\lambda})=\cL_{\varphi,\lambda}$, and 
\item  $\left.\tfrac{d}{dX}\right\mid _{X=0}a_\ell(\cF)=- \xi \cL\cdot  \cL_{\psi,\ell}$. 
\end{enumerate}
\end{prop}
\begin{proof} If $\ell$ divides  $N$ then it also divides the conductor of $\psi$, hence $\rho_f^{\I_\ell}=\{0\}$ and $a_\ell=0$. 
Then $a_\ell(\cF)=0$ as well, since by \cite[\S6]{dim-durham} all classical specializations of a Hida family $\cF$ share the same local type at $\ell$. 

 (i) Letting $\pi\mid\gp$ be the prime of $H$ given by $\iota_p$, Class Field Theory yields  an exact sequence
\[ 0\to \Hom(\G_H,\bar\Q_p)\to \Hom\left(\prod_{h \in \Gal(H/K)}
 (\cO_{H,h(\pi)}^\times \times \cO_{H,h(\bar\pi)}^\times) \times  H_{\lambda}^\times ,\bar\Q_p\right)
\to \Hom(\cO_H[\tfrac{1}{\lambda}]^\times,\bar\Q_p),\]
sending $\eta_{\varphi\mid \G_H}$ to 
 $\left( (\varphi^{-1}(h)(1,\cS_{\varphi})\log_p)_{h \in \Gal(H/K)},\eta_{\varphi}(\Frob_{\lambda})\cdot \ord_{\lambda}\right)$. 
The triviality on $u_{\lambda}\in \cO_H[\tfrac{1}{\lambda}]^\times$ yields by \eqref{unit lambda} and \eqref{inert-L-inv} the desired equality
\[\log_p(u_{\varphi,\lambda}) + \cS_{\varphi}   \log_p(\tau(u_{\varphi,\lambda}))+ \eta_{\varphi}(\Frob_{\lambda})  \ord_\lambda(u_\lambda)=0.\]

(ii) As $\rho_f(g)= \left(\begin{smallmatrix} 0 & \psi(g\tau) \\  \psi(\tau g)  & 0 \end{smallmatrix} \right)$, for all $g\in \G_\Q \backslash \G_K$, one has 
  \[ \rho_{\cF}(g^2)=\rho_{\cF}(g) \rho_{\cF}(g)\equiv\left(\begin{matrix} \ast & \psi(g \tau) \cdot \tr (\rho_{\cF}(g)) \\
\psi(\tau g)   \tr (\rho_{\cF}(g)) &  \ast  \end{matrix} \right) \mod{X^2}. \]
Comparing the expression of the upper right coefficient of $\rho_{\cF}(g^2)$ with Proposition~\ref{infinitesimal-rhoF}, one finds 
  \begin{equation}\label{lasteq}
  \tr \rho_{\cF}(g) =-\xi\cL \psi(\tau g ) \eta_{\varphi}(g^2) X \pmod X^2, \text{ for all } g\in \G_\Q \backslash \G_K. 
  \end{equation}

Applying this to $g=\Frob_\ell$, so that  $g^2=\Frob_{\lambda}$,    one deduces that 
\[
\left.\tfrac{d}{dX}\right\mid _{X=0}a_\ell(\cF)=-\xi\cL \psi(\tau \Frob_\ell ) \eta_{\varphi}(\Frob_{\lambda}).
\]
The claim then follows  from (i), in view of \eqref{psi-invariant}. 
\end{proof}

\subsection{Computation at  ramified primes}
Finally, we turn to the case when   $(\ell)=\gl^2$ ramifies  in $K$. 
We have seen that $a_\ell=\psi(\gl)$ and the aim of this section is to express infinitesimally $a_\ell(\cF)$ in terms of logarithms of  $\ell$-units.  
If $\ell D \mid N$, applying the same argument as in  Proposition \ref {prop-inert} yields $a_\ell(\cF)=0$.
Henceforth we assume that $\ell D \nmid N$, and recall that 
$\gamma\in\G_\Q\backslash \G_K$ is an  element of an inertia subgroup of $\G_\Q$ at $\ell$  whose image in $\Gal(H/\Q)$ generates the order two subgroup $\Gal(H_\lambda/\Q_\ell)=
\I(H_\lambda/\Q_\ell)$ (see \S\ref{l-units}).

\begin{prop}\label{prop-ram} Assume that $\ell\mid D$, but  $\ell D \nmid N$. Then
\begin{enumerate}
\item  $\eta_{\gp}(\Frob_{\gl})=\eta_{\bar\gp}(\Frob_{\gl})=\cL_{\gl}=-\tfrac{1}{2}\log_p(\ell)$ and 
$\eta_{\varphi}(\Frob_{\ell})=\cL_{\varphi,\lambda}$,  
\item  $\left.\tfrac{d}{dX}\right\mid _{X=0}a_\ell(\cF)=\psi(\gl)\cdot\left(\tfrac{\log_p(\ell)}{2\log_p(1+p^{\nu})} -\xi \cL\cdot  \cL_{\psi,\ell}\right)$. 
\end{enumerate}
\end{prop}
\begin{proof}

(i) The  restrictions of $\eta_{\gp}$, $\eta_{\bar\gp}$ and $\eta_{\varphi}$  (see \eqref{trivial-at-ell}) to the inertia subgroup  at $\gl$ belong 
to $\Hom(\cO_{K_\gl}^\times, \bar\Q_p)=\{0\}$. Therefore their values at $\Frob_{\gl}$ are well defined (depending on the choice of
$\lambda\mid \gl$ for $\eta_{\varphi}$) and can be computed using Class Field Theory exactly as in Propositions \ref{prop-split} and \ref{prop-inert}. 

(ii) By (i) and Proposition~\ref{infinitesimal-rhoF}, 
$\rho_\cF(\gamma) \mod{X^2} = \rho_f(\gamma)=  \left( \begin{smallmatrix} 0& \psi(\gamma\tau) \\ \psi(\tau\gamma) & 0  \end{smallmatrix}\right)$ has a fixed line  generated by the vector $e_1+ \psi(\tau\gamma) e_2$.  Using again (i) and  Proposition~\ref{infinitesimal-rhoF}, together with \eqref{psi-invariant},   one finds 
\begin{align*}
&a_\ell(\cF) = (1,0)
\rho_{\cF}(\Frob_{\ell}) \left(\begin{smallmatrix} 1 
 \\   \psi(\tau\gamma)  \end{smallmatrix}\right)
  \\ 
&\equiv \psi(\gl)  (1,0) \left(\left(\begin{smallmatrix} 1 & 0
 \\ 0  & 1 \end{smallmatrix}\right) - \left(\begin{smallmatrix} (\bar\cL\eta_\gp+\cL\eta_{\bar\gp})(\Frob_{\gl}) & \xi \cL \eta_{\varphi}(\Frob_{\ell}) 
 \\  \xi^{-1} \bar\cL  \cS_{\bar\varphi} \bar\eta_{\varphi}(\Frob_{\ell})  &  (\cL\eta_\gp+\bar\cL\eta_{\bar\gp})(\Frob_{\gl}) \end{smallmatrix}\right) X \right)
 \left(\begin{smallmatrix} 1  \\ \psi(\tau\gamma) \end{smallmatrix}\right) \\  
& \equiv \psi(\gl)- \psi(\gl) \cdot\left( (\cL+\bar\cL)\cdot \cL_{\gl}+\xi \psi(\tau\gamma) \cL\cdot  \cL_{\varphi,\lambda}\right) 
\cdot X \pmod{X^2}. \qedhere
\end{align*}
\end{proof}

As  $\cF$ has  local type $1\oplus \varepsilon_K$ at $\ell$, one deduces that  
 $\tfrac{a_\ell(\cF)\cdot a_\ell(\cF\otimes\varepsilon_K)}{\psi(\gl)^2}=
\left(\tfrac{\det\rho_{\cF}}{\det\rho_f}\right)(\Frob_{\ell})$ which concords with Proposition~\ref{prop-ram}(ii). 

\subsection{On the $q$-expansion of CM  families specializing to $f$}

Denote by  $ \cC\ell_K^{(p)}(\gp^{\infty})$  the $p$-primary part of the ray class group of $K$ of conductor $\gp^{\infty}$. 
The torsion-free quotient $\cC\ell_K^{(p)}(\gp^{\infty})_{/\tor}$ is the Galois group of the
 unique $\gp$-ramified $\Z_p$-extension of $K$. We introduced in \cite[\S3]{betina-dimitrov} a $p$-adic avatar $\G_K \twoheadrightarrow \cC\ell_K^{(p)}(\gp^{\infty})_{/\tor} \to \bar\Z_p^{\times}$  of a Hecke character  of infinity type $(1, 0)$, further used to define of a universal character \[\G_K \twoheadrightarrow \cC\ell_K^{(p)}(\gp^{\infty})_{/\tor}   \hookrightarrow \bar\Z_p^{\times} \lsem \cC\ell_K^{(p)}(\gp^{\infty})_{/\tor}  \rsem^{\times}\] 
 interpolating $p$-adically Hecke characters of $K$ whose infinity type belongs to $\{(k,0); k\in \Z \}$. Its localization  yields a character  $\chi_\gp:\G_K \to \Lambda^{\times}$ such that   
 \begin{equation}\label{univ10}
\chi_\gp\equiv 1 - \frac{\eta_{\gp}}{\log_p(1+p^{\nu})}X\pmod{X^2}. 
\end{equation}
It is easy to check that $\chi_\gp\bar\chi_\gp$ extends to $\G_\Q$ and equals the universal cyclotomic character.

There exist two CM families $\Theta_{\psi}$ and $\Theta_{\bar\psi}$ specializing to $f$ in weight one and such that their attached $\Lambda$-adic representation is given by $\Ind^\Q_K \psi \chi_\gp$ and $\Ind^\Q_K \bar\psi \chi_\gp$ respectively, hence in particular their Fourier coefficients vanish at all primes $\ell$ inert in $K$. Since the  CM-line given by $y=0$ in Figure~\ref{graph} corresponds to the tangent space of $\Theta_{\psi}$, if one lets 
 $\cL=0$ in Proposition \ref{prop-split}, then  one obtains that  for any prime $\ell\ne p$ splitting in $K$
\[
 \left.\tfrac{d}{dX}\right\mid _{X=0}a_\ell(\Theta_{\psi})= - \frac{\psi(\gl)\cL_{\gl}+\psi(\bar{\gl})\cL_{\bar{\gl}}}{\log_p(1+p^{\nu})},  \text{ and }
\left.\tfrac{d}{dX}\right\mid _{X=0}a_p(\Theta_{\psi}) = \frac{\psi(\gp)  \cL_\gp}{\log_p(1+p^{\nu})}. 
\]

Finally,  for any $\ell \mid D$, one has (by letting  $\cL=0$ in  Proposition \ref{prop-ram}):
\[
 \left.\tfrac{d}{dX}\right\mid _{X=0}a_\ell(\Theta_{\psi})=\tfrac{\psi(\gl)\log_p(\ell)}{2\log_p(1+p^{\nu})}.  
\]

\subsection{Generalized eigenforms at weight $1$ CM points of the eigencurve}
Recall from \S\ref{sec-irregular} that one can   attach to $f$ a generalized eigenspace 
$S_{\sw(f)}^{\dag}\lsem f \rsem=S_{\sw(f)}^{\dag}\lsem \gm \rsem$. 
One clearly has $S_{\sw(f)}^{\dag}[\gm]=\bar\Q_p\cdot f$ and we have already observed that  classical subspace of
$S_{\sw(f)}^{\dag}\lsem f \rsem$ has a basis $\{f,\theta_{\psi}\}$ whose elements belong to $S_{\sw(f)}^{\dag}[\gm^2]$. 

Under the running assumption \eqref{generic-CM}, it is shown in  \cite{betina-dimitrov} that $\dim S_{\sw(f)}^{\dag}\lsem f \rsem=\dim S_{\sw(f)}^{\dag}[\gm^2]=4$ and hence the space $S_{\sw(f)}^{\dagger} \lsem f \rsem_0=S_{\sw(f)}^\dagger[\gm^2]_0$
of genuine overconvergent generalized  eigenforms defined in  \S\ref{sec-irregular} is two-dimensional. 
We consider the following forms in $S_{\sw(f)}^{\dagger} \lsem f \rsem$:
\begin{align}
f^\dag_{\cF} &=  \log_p(1+p^{\nu})\cdot  \left.\tfrac{d}{dX}\right\mid _{X=0}\left(\cF\otimes\varepsilon_K -\cF\right), \text{ and }\\
  f^\dag_{\Theta} &=  \log_p(1+p^{\nu})\cdot  \left.\tfrac{d}{dX}\right\mid _{X=0}\left(\Theta_{\bar\psi} -\Theta_{\psi}\right).
\end{align}

A  computation based on the  $q$-expansions Principle (see Proposition~\ref{q-expprinciple})
and Proposition~\ref{bad-Hecke-ops},  implies the following linear relation 
in $S_{\sw(f)}^{\dagger} \lsem f \rsem$: 
\[
 \log_p(1+p^{\nu})^2\cdot\left.\tfrac{d}{dX}\right\mid _{X=0}\left(
 \bar\cL \Theta_{\psi}+ \cL \Theta_{\bar\psi}- \tfrac{(\cL+\bar\cL)}{2}(\cF+\cF\otimes \varepsilon_K)
\right)
=\frac{\cLm(\varphi)\cdot \cLm(\bar\varphi)}{\cLm(\varphi)+\cLm(\bar\varphi)}(f-\theta_\psi).
  \]

\begin{proof}[Proof of Theorem \ref{main-q-exp}]
Parts (i)-(iii) are direct consequence of Propositions \ref{prop-split}, \ref{prop-inert} and \ref{prop-ram}, in view of the $q$-expansion Principle  (see Definition~\ref{ell-L-inv} for the definition of $ \cL_{\psi,\ell}$). Part (iv) results from the well-known  relations between the abstract Hecke operators. 
\end{proof}

\begin{proof}[Proof of Theorem \ref{thm-cross-ratio}] In the basis $(e_1, e_2)$ from \S\ref{CM-case}, $\rho_f(\tau)$ fixes the vector 
$e_1+e_2$, while the ordinary line of $\Theta_\psi$, resp. $\Theta_{\bar\psi}$ is spanned by $e_1$, resp. $e_2$. 
Finally by Proposition \ref{infinitesimal-rhoF} the ordinary line of $\cF$ is spanned residually by 
 $\xi\cdot e_1+ e_2$, allowing us to compute the desired  cross-ratio as follows: 
 \[\left[e_1+ e_2,  \xi\cdot e_1+ e_2;  e_1, e_2\right] =\frac{\left|\begin{smallmatrix} 1 & 1 \\ 1 & 0 \end{smallmatrix}\right|
 \cdot \left|\begin{smallmatrix} \xi & 0 \\ 1 & 1 \end{smallmatrix}\right|}
 {\left|\begin{smallmatrix} 1 & 0 \\ 1 & 1 \end{smallmatrix}\right|\cdot \left|\begin{smallmatrix} \xi & 1 \\ 1& 0 \end{smallmatrix}\right|}=\xi. \qedhere\]
\end{proof}

\begin{rem}\label{rem:f-dagger} [The point $f_\psi^{\dagger}$ in Figure~\ref{graph}] 
The Galois representation $\rho$ from Proposition \ref{prop-tangent}, whose semi-simplification 
is given by $\rho_{f|\G_K}$, has a unique ordinary line and the tangent space of the corresponding deformation problem is represented in Figure~\ref{graph}. It is clear that it has a unique CM deformation given by $\Theta_\psi$ 
and that both $\cF$ and its twist $\cF\otimes \varepsilon_K$ yield the same non-CM deformation. 
Then $f^\dag_{\psi}=  \left.\tfrac{d}{dX}\right\mid _{X=0}\left(\cF -\Theta_{\psi} \right)$  is a  genuine overconvergent generalized eigenform, and its Fourier coefficient at a prime $\ell$ is given by 
\[
a_\ell(f^\dag_{\psi})= \cL\cdot \begin{cases} 
(\psi(\gl) -\psi(\bar\gl)) \cdot(\cL_{\gl}-  \cL_{\bar{\gl}})&, \text{ if } \ell \ne p  \text{ splits as  }  \gl \cdot \bar{\gl}, \\
- \xi \cdot   \cL_{\psi,\ell}& , \text{ if } \ell \nmid N \text{ is inert   } \\
0 &, \text{ if } \ell \mid N \text{ is inert   } \\
\psi(\gp) \cdot \cLm(\bar\varphi)& , \text{ if } \ell=p, \\-
\psi(\gl) \cdot \xi \cdot \cL_{\psi,\ell} &, \text{ if } \ell \mid D. 
\end{cases}
\]
\end{rem}

\bibliographystyle{siam}

\end{document}